\documentclass[10pt,a4paper,reqno,twoside]{amsart}

\newcommand{\usepackageifexists}[1]{%
    \IfFileExists{#1.sty}{\usepackage{#1}}%
       {\GenericInfo{taglia}{Il package #1 non esiste.}}}

\usepackage{verbatim}
\usepackage{amssymb}
\usepackage{upref}
\usepackage{amsmath}
\usepackage{amsthm}
\usepackage{amsfonts}
\usepackage{latexsym}
\usepackage{amssymb}
\usepackage{color}
\usepackage{amsmath}
\usepackage{amsfonts}
\usepackage{amsthm}
\usepackage{amssymb}
\usepackage{amsmath,amsfonts,amscd,bezier}
\usepackage[latin1]{inputenc}

\usepackage{enumerate}

\theoremstyle{plain}
\newtheorem{Thm}{Theorem}
\newtheorem*{Thm*}{Theorem}

\newtheorem{Prop}{Proposition}
\newtheorem{Lem}[Prop]{Lemma}

\theoremstyle{definition}
\newtheorem{Def}{Definition}

\theoremstyle{remark}
\newtheorem{Rem}{Remark}

\newtheorem{Example}{Example}

\allowdisplaybreaks

\newcommand{\N}{\mathbb{N}}

\newcommand{\R}{\mathbb{R}}
\newcommand{\Rn}{{\R^n}}

\renewcommand{\doteq}{:=}

\def\<#1\>{\left\langle#1\right\rangle }

\begin{document}

\title[Semi-linear structural damped waves]
          {Semi-linear structural damped waves}
\author[M. D'Abbicco, M. Reissig]%
    {Marcello D'Abbicco, Michael Reissig}

\address{Marcello D'Abbicco, Department of Mathematics, University of Bari, Via E. Orabona 4 - 70125 BARI -
ITALY}

\address{Michael Reissig, Faculty for Mathematics and Computer
Science, Technical University Bergakademie Freiberg, Pr\"uferstr.9 -
09596 FREIBERG - GERMANY}
\begin{abstract}
We study the Cauchy problem for the semi-linear structural damped
wave equation with source term
\[ u_{tt}-\triangle u+ \mu(-\Delta)^\sigma u_t=f(u), \quad u(0,x)=u_0(x), \quad u_t(0,x)=u_1(x), \]
with $\sigma \in (0,1]$ in space dimension~$n\geq2$ and with a
positive constant $\mu$. We are interested in the influence of
$\sigma$ on the critical exponent $p_{crit}$ in $|f(u)|\approx
|u|^p$. This critical exponent is the threshold between global
existence in time of small data solutions and blow-up behavior for
some suitable range of $p$. Our results are optimal for
parabolic-like models $\sigma \in (0,1/2]$.
\end{abstract}



\maketitle

\maketitle


\section{Introduction}\label{sec:intro}
\noindent In this paper, we consider the Cauchy problem
\begin{equation}
\label{eq:diss}
\begin{cases}
u_{tt}-\Delta u+\mu(-\Delta)^\sigma u_t=f(u), & t\geq0, \ x\in\R^n,\\
u(0,x)=u_0(x), \\
u_t(0,x)=u_1(x),
\end{cases}
\end{equation}
in space dimension~$n\geq1$ with~$\sigma\in(0,1]$, where~$\mu>0$ is a constant and the nonlinear term satisfies
\begin{equation}\label{eq:disscontr}
f(0)=0, \qquad |f(u)-f(v)|\lesssim |u-v|(|u|+|v|)^{p-1},
\end{equation}
for a given $p>1$. The case~$\sigma=1$ corresponds to the wave equation with \emph{visco-elastic} damping~\cite{Shibata}, whereas for~$\sigma\in(0,1)$ we are dealing with a \emph{structural damping}. Our aim is to prove global existence result of small data solution for~$p>\overline{p}(\sigma,n)$ in low space dimension. In order to do this we derive suitable estimates for the corresponding linear problem and we prove that the solution to the semilinear one satisfies the same estimates. Moreover, we also prove that our exponent $\overline{p}(\sigma,n)$ is \emph{critical} for~$\sigma\in(0,1/2]$.


\subsection{Main results} \label{Sec1.1}

The following definition would fix the energy space for the data in our statements.
\begin{Def}
We define
\begin{equation}\label{eq:Dm}
\mathcal{D}_m^k \doteq (L^1\cap H^{k,m}) \times (L^1\cap L^m),\qquad  \|(v_0,v_1)\|_{\mathcal{D}_m^k} \doteq \|v_0\|_{L^1}+\|v_0\|_{H^{k,m}}+\|v_1\|_{L^1}+\|v_1\|_{L^m}
\end{equation}
for~$m\in (1,2]$ and~$k\geq0$.
\end{Def}
For the sake of clarity, we will denote by $L^1\cap L^m$ the space~$\mathcal{D}_m^0= (L^1\cap L^m)\times (L^1\cap L^m)$.
We remark that~$\mathcal{D}_2^1 = (L^1\cap H^1) \times (L^1\cap L^m)$ corresponds to the \emph{classical energy space} $H^1\times L^2$
with additional~$L^1$ regularity.
\\
Distinguishing four models we can now present our main results.
\begin{Thm}\label{Thm:halfnl}
Let~$\sigma=1/2$ in~\eqref{eq:diss}. Let $n=2,3,4$ and let $m\in(1,2]$. Let $p\in [m,n/(n-m)]$ be such that
\begin{equation}\label{eq:phalf}
p > 1+ \frac2{n-1}.
\end{equation}
Then there exists~$\epsilon>0$ such that for any~$(u_0,u_1)\in\mathcal{D}_m^1$ with $\|(u_0,u_1)\|_{\mathcal{D}_m^1}<\epsilon$, there exists a unique
solution ~$u \in \mathcal{C}([0,\infty),H^{1,m})\cap\mathcal{C}^1([0,\infty),L^m)$ to~\eqref{eq:diss}.
\\
Moreover, the solution and its energy based on the~$L^m$ norm satisfy the estimates
\begin{align}
\label{eq:decu}
\|u(t,\cdot)\|_{L^m}
    & \leq C (1+t)^{1-n(1-\frac1m)} \|(u_0,u_1)\|_{\mathcal{D}_m^1}, \\
\label{eq:decgradu}
\|\bigl(\nabla u(t,\cdot), u_t(t,\cdot)\bigr)\|_{L^m}
    & \leq C (1+t)^{-n(1-\frac1m)} \|(u_0,u_1)\|_{\mathcal{D}_m^1},
\end{align}
where~$C>0$ does not depend on the data.
\end{Thm}
The exponent which is given by~\eqref{eq:phalf} is \emph{critical} (see later, Theorem~\ref{Thm:blow}).
\begin{Rem} \label{Remmixing} We want to underline that the results from Theorem \ref{Thm:halfnl} base on the mixing of different
regularity for the data, where the data do not necessarily belong to the classical energy space $H^1 \times L^2$.
\end{Rem}
\begin{Thm}\label{Thm:Shinl}
Let~$\sigma=1$ in~\eqref{eq:diss}. Let $n\geq2$ and let $p\in [2,n/(n-4)]$ be such that
\begin{equation}\label{eq:pShi}
p > 1+ \frac3{n-1}.
\end{equation}
Then there exists~$\epsilon>0$ such that for any~$(u_0,u_1)\in\mathcal{D}_2^2$ with $\|(u_0,u_1)\|_{\mathcal{D}_2^2}<\epsilon$, there exists a unique
solution~$u \in \mathcal{C}([0,\infty),H^2)\cap\mathcal{C}^1([0,\infty),L^2)$ to~\eqref{eq:diss}.
\\
Moreover, the solution, its first derivative in time, and its derivatives in space up to the second order, satisfy the decay estimates
\begin{align}
\label{eq:uShi}
\|u(t,\cdot)\|_{L^2}
    & \leq \begin{cases}
        C (1+t)^{-\frac{n-2}4} \, \|(u_0,u_1)\|_{L^1\cap L^2} & \text{if~$n\geq3$,}\\
        C \log (e+t) \, \|(u_0,u_1)\|_{L^1\cap L^2} & \text{if~$n=2$,}
        \end{cases} \\
\label{eq:utShi}
\|u_t(t,\cdot)\|_{L^2}
    & \leq C (1+t)^{-\frac{n}4} \, \|(u_0,u_1)\|_{L^1\cap L^2},\\
\label{eq:uxShi}
\|\nabla u(t,\cdot)\|_{L^2}
    & \leq C (1+t)^{-\frac{n}4} \|(u_0,u_1)\|_{\mathcal{D}_2^1},\\
\label{eq:uxxShi}
\|\nabla^2 u(t,\cdot)\|_{L^2}
    & \leq C (1+t)^{-\frac{n+2}4} \|(u_0,u_1)\|_{\mathcal{D}_2^2},
\end{align}
where~$C>0$ does not depend on the data.
\end{Thm}
\begin{Rem} \label{higherorder} We want to underline, that the results from Theorem \ref{Thm:Shinl} base on the one hand on the use of
higher order of regularity, namely, second order in space, and on the other hand on the mixing of different regularity for the data, where the data belong
to the classical energy space, too.
\end{Rem}
\begin{Thm}\label{Thm:parnl}
Let~$\sigma\in(0,1/2)$ in~\eqref{eq:diss}. Let $n=2,3,4$ and let $p\in [2,n/(n-2)]$ be such that
\begin{equation}\label{eq:ppar}
p > 1+ \frac2{n-2\sigma}.
\end{equation}
Then there exists~$\epsilon>0$ such that for any~$(u_0,u_1)\in\mathcal{D}_2^1$ with $\|(u_0,u_1)\|_{\mathcal{D}_2^1}<\epsilon$, there exists a unique
solution~$u \in \mathcal{C}([0,\infty),H^1)\cap\mathcal{C}^1([0,\infty),L^2)$ to~\eqref{eq:diss}.
\\
Moreover, the solution, and its first derivatives in time and space satisfy the decay estimates
\begin{align}
\label{eq:upar}
\|u(t,\cdot)\|_{L^2}
    & \lesssim (1+t)^{-\left(\frac{n}4-\sigma\right)\,\frac1{1-\sigma}} \|(u_0,u_1)\|_{L^1\cap L^2}, \\
\label{eq:utpar}
\|u_t(t,\cdot)\|_{L^2}
    & \lesssim (1+t)^{-1} \, \|(u_0,u_1)\|_{\mathcal{D}_2^1},\\
\label{eq:uxpar}
\|\nabla u(t,\cdot)\|_{L^2}
    & \lesssim (1+t)^{-\left(\frac{n+2}4-\sigma\right)\,\frac1{1-\sigma}} \, \|(u_0,u_1)\|_{\mathcal{D}_2^1},
\end{align}
where~$C>0$ does not depend on the data.
\end{Thm}
The exponent which is given by~\eqref{eq:ppar} is \emph{critical} (see later, Theorem~\ref{Thm:blow}).
\begin{Rem} \label{classicalenergymixing} We want to underline, that the results from Theorem \ref{Thm:parnl} base on the mixing of the
regularity for the data between $(H^1,L^2)$ and $L^1$. The data belong to the classical energy space, too.
\end{Rem}
\begin{Thm}\label{Thm:hypnl}
Let~$\sigma\in(1/2,1)$ in~\eqref{eq:diss}. Let $n\geq2$ and let $p\in [2,n/(n-4\sigma)]$ be such that
\begin{equation}\label{eq:phyp}
p > 1+ \frac{1+2\sigma}{n-1}.
\end{equation}
Then there exists~$\epsilon>0$ such that for any~$(u_0,u_1)\in\mathcal{D}_2^{2\sigma}$ with $\|(u_0,u_1)\|_{\mathcal{D}_2^{2\sigma}}<\epsilon$, there
exists a unique solution~$u \in \mathcal{C}([0,\infty),H^{2\sigma})\cap\mathcal{C}^1([0,\infty),L^2)$ to~\eqref{eq:diss}.
\\
Moreover, the solution, its first derivatives in time and space, and its derivative in space of fractional order~$2\sigma$, satisfy the decay estimates
\begin{align}
\label{eq:uhyp} \|u(t,\cdot)\|_{L^2}
    & \lesssim (1+t)^{-\frac{n-2}{4\sigma}} \|(u_0,u_1)\|_{L^1\cap L^2}, \\
\label{eq:uthyp}
\|u_t(t,\cdot)\|_{L^2}
    & \leq C (1+t)^{-\frac{n}{4\sigma}} \, \|(u_0,u_1)\|_{\mathcal{D}_2^{2(1-\sigma)}},\\
\label{eq:uxhyp}
\|\nabla u(t,\cdot)\|_{L^2}
    & \leq C (1+t)^{-\frac{n}{4\sigma}} \|(u_0,u_1)\|_{\mathcal{D}_2^1},\\
\label{eq:uxxhyp}
\|u(t,\cdot)\|_{\dot{H}^{2\sigma}}
    & \leq C (1+t)^{-\frac{n-2}{4\sigma}-1} \|(u_0,u_1)\|_{\mathcal{D}_2^{2\sigma}},
\end{align}
where~$C>0$ does not depend on the data.
\end{Thm}
\begin{Rem} \label{higherorderclassicalmixing} We want to underline, that the results from Theorem \ref{Thm:hypnl} base on the one hand
on the use of higher order of regularity, namely, fractional order $2\sigma$ in space, and on the other hand on the mixing of different regularity for the
data, where the data belong to the classical energy space, too.
\end{Rem}
%


\subsection{A comparison with the classical damped wave equation} \label{Sec1.2}

Let us compare the results from Section \ref{Sec1.1}  with some known results for the \emph{classical damped wave equation}.  Some new effects appear for
semi-linear structural damped waves which we will explain in Section \ref{Sec1.3}. If we set~$\sigma=0$ in~\eqref{eq:diss}, then we get
\begin{equation}
\label{eq:classic}
\begin{cases}
u_{tt}-\Delta u+\mu u_t=f(u), & t\geq0, \ x\in\R^n,\\
u(0,x)=u_0(x), \\
u_t(0,x)=u_1(x),
\end{cases}
\end{equation}
for which many results are known concerning global existence of small data solutions and sharp decay estimates.
\\
In particular, let~$n\leq4$ and let
\[ p\in\begin{cases}
(1+2/n,\infty) & \text{if~$n=1,2$,}\\
[2,3] & \text{if~$n=3$,}\\
\{2\} & \text{if~$n=4$.}
\end{cases}\]
Then there exists~$\epsilon>0$ such that for any initial data~$(u_0,u_1)\in (L^1\cap H^1) \times (L^1\cap L^2)$ which
satisfy~$\|u_0\|_{L^1\cap H^1}+\|u_1\|_{L^1\cap L^2}\leq\epsilon$, there exists a unique
solution~$u\in\mathcal{C}([0,\infty),H^1)\cap\mathcal{C}^1([0,\infty),L^2)$ to~\eqref{eq:classic} (see~\cite{IO}). Moreover, such a solution
and its first derivatives with respect to $t$ and $x$ satisfy the same decay estimates of the linear problem~\cite{Matsu}
(i.e. \eqref{eq:classic} with~$f\equiv0$), that is,
\begin{align}
\label{eq:uclassic}
\|u(t,\cdot)\|_{L^2} & \lesssim (1+t)^{-\frac{n}4} \, \|(u_0,u_1)\|_{L^1\cap L^2}, \\
\label{eq:utclassic}
\|u_t(t,\cdot)\|_{L^2} & \lesssim (1+t)^{-\frac{n}4-1} \, \bigl(\|u_0\|_{L^1\cap H^1}+\|u_1\|_{L^1\cap L^2}\bigr),\\
\label{eq:uxclassic} \|\nabla u(t,\cdot)\|_{L^2} & \lesssim (1+t)^{-\frac{n}4-\frac12} \, \bigl(\|u_0\|_{L^1\cap H^1}+\|u_1\|_{L^1\cap L^2}\bigr).
\end{align}
Moreover, the exponent~$1+2/n$ is \emph{critical}. In particular, if one set~$f(u)=|u|^p$ in~\eqref{eq:classic} and if the data are
in~$\mathcal{C}_0^\infty$ and satisfy~$\int_{\R^n} u_j(x)\,dx>0$, for~$j=0,1$, then there exists no global solution to~\eqref{eq:classic} for
any~$p\leq1+2/n$ and for any~$n\geq1$.


\subsection{Overview of the four models}\label{sec:overview} \label{Sec1.3}

We notice that the properties of solutions to our model~\eqref{eq:diss} changes completely from~$\sigma\in(0,1/2]$ to~$\sigma\in[1/2,1]$. Therefore we
propose to distinguish between \emph{parabolic type} ($\sigma\in(0,1/2)$) and \emph{hyperbolic type} ($\sigma\in(1/2,1]$) models, having in mind that the
classical damped wave equation is a \emph{parabolic} model, whereas the visco-elastic damped wave equation is a \emph{hyperbolic} model.
We consider the case $\sigma=1/2$ as a critical case.
\begin{itemize}
\item The structure of the case~$\sigma=1/2$ is very special and easy to manage. This simplicity allows us easily to state a result which
also includes an energy based on~$L^m$ norm, with~$m\in(1,2]$. We remark that the first derivatives of the solution with respect to time and to space have
the same decay rate. This is a new effect with respect to the case~$\sigma=0$, for which the decay rate in~\eqref{eq:utclassic} is better than the one
in~\eqref{eq:uxclassic}.
\item Dealing with the case~$\sigma=1$, completely new effects arise with respect to the cases~$\sigma=0$ or~$\sigma=1/2$.
In particular, we see that the first derivatives of the solution with respect to time and space  have the same decay rate, as in the case~$\sigma=1/2$. On
the other hand, the estimate for the first derivative in time~\eqref{eq:utShi} requires less regularity for the data comparing with respect to the estimate
for the first derivative in space~\eqref{eq:uxShi}. This property is new in comparison with respect to both cases~$\sigma=0,1/2$. Moreover, we can also
obtain a decay estimate for the space derivatives up to the second order if we assume $H^2$ regularity for~$u_0$ with no need of additional regularity
for~$u_1$. This property is very useful to deal with semilinear problems (see Remark~\ref{Rem:u1L2}).
\item The case~$\sigma\in(0,1/2)$ interpolates the cases~$\sigma=0$ and~$\sigma=1/2$. In particular, the \emph{critical exponent}
$1+2/(n-2\sigma)$ and the decay rates for the solution and its first derivatives are continuous with respect to~$\sigma$ for~$\sigma\in[0,1/2]$.
We remark that the decay rate in~\eqref{eq:utpar} is better than the one in~\eqref{eq:uxpar}, but the regularity of the data is the same.
\item The case~$\sigma\in(1/2,1)$ interpolates the cases~$\sigma=1/2$ and~$\sigma=1$. In particular,
the exponent $1+(1+2\sigma)/(n-1)$, the decay rates for the energy of the solution, and the regularity required on the data are continuous with respect
to~$\sigma$ for~$\sigma\in[1/2,1]$. We remark that in this case an estimate on the fractional derivative of order~$2\sigma$ of the solution appears.
\end{itemize}

\begin{Rem}\label{Rem:ranges}
We have different ranges for~$m$ and~$p$ for which we can apply Theorem~\ref{Thm:halfnl}.
\begin{itemize}
\item Let~$n=2$. Then we can apply Theorem~\ref{Thm:halfnl} for any~$m\in (4/3,2]$ and~$p\in(3,2/(2-m)]$.
\item Let~$n=3$. Then we can apply Theorem~\ref{Thm:halfnl} for any~$m\in (3/2,2]$ and~$p\in(2,3/(3-m)]$.
\item Let~$n=4$. Then we can apply Theorem~\ref{Thm:halfnl} for any~$m\in (5/3,2]$ and~$p\in[m,4/(4-m)]$, or for any~$m\in(8/5,5/3]$ and~$p\in(5/3,4/(4-m)]$.
\item We can not apply Theorem~\ref{Thm:halfnl} if~$n=5$. The set of admissible $p$ is empty.
\end{itemize}
In Theorem~\ref{Thm:Shinl} we have the following ranges for~$p$:
\[ p \in
\begin{cases}
(4,\infty) & \text{if~$n=2$,}\\
(5/2,\infty) & \text{if~$n=3$,}\\
(2,\infty) & \text{if~$n=4$,}\\
[2,5] & \text{if~$n=5$,}\\
[2,3] & \text{if~$n=6$,}\\
[2,7/3] & \text{if~$n=7$,}\\
\{2\} & \text{if~$n=8$.}
\end{cases}
\]
The set is empty for~$n\geq9$. In Theorem~\ref{Thm:parnl} we have the following ranges for~$p$:
\[ p \in
\begin{cases}
(1+1/(1-\sigma),\infty) & \text{if~$n=2$,}\\
[2,3] & \text{if~$n=3$,}\\
\{2\} & \text{if~$n=4$.}
\end{cases}
\]
The set is empty for~$n\geq5$. In Theorem~\ref{Thm:hypnl} we have the following ranges for~$p$:
\[ p \in
\begin{cases}
(2+2\sigma,\infty) & \text{if~$n=2$,}\\
((3+2\sigma)/2,\infty) & \text{if~$n=3$ and~$\sigma\in[3/4,1)$,}\\
((3+2\sigma)/2,3/(3-4\sigma)] & \text{if~$n=3$ and~$\sigma\in(1/2,3/4)$,}\\
[2,1/(1-\sigma)] & \text{if~$n=4$,}\\
[2,5/(5-4\sigma)] & \text{if~$n=5$ and $\sigma\in(5/8,1)$,}\\
[2,3/(3-2\sigma)] & \text{if~$n=6$ and $\sigma\in(3/4,1)$,}\\
[2,7/(7-4\sigma)] & \text{if~$n=7$ and $\sigma\in(7/8,1)$.}
\end{cases}
\]
The set is empty for~$n\geq8$.
\end{Rem}


\section{Linear decay estimates}\label{sec:linear}
\noindent In this section our aim is to derive decay estimates for the solution and some of its derivatives to the linear Cauchy problem
\begin{equation}
\label{eq:lin}
\begin{cases}
v_{tt}-\Delta v+\mu(-\Delta)^\sigma v_t=0, & t\geq0, \ x\in\R^n,\\
v(0,x)=v_0(x), \\
v_t(0,x)=v_1(x),
\end{cases}
\end{equation}
which corresponds to~\eqref{eq:diss} when~$f\equiv0$.


\subsection{The case~$\sigma=1/2$}\label{sec:half}

In this case we have the following statement.
\begin{Thm}\label{Thm:halflin}
Let~$\sigma=1/2$ in~\eqref{eq:lin}. Let $n\geq1$ and let $m\in(1,2]$. Let~$(v_0,v_1)\in\mathcal{D}_m^1$. Then the solution to~\eqref{eq:lin} and its energy based on the~$L^m$ norm satisfy the $(L^1\cap L^m)-L^m$ estimates
\begin{align}
\label{eq:v}
\|v(t,\cdot)\|_{L^m}
    & \lesssim (1+t)^{1-n(1-\frac1m)} \, \|(v_0,v_1)\|_{L^1 \cap L^m}\,, \\
\label{eq:gradv}
\|\bigl(\nabla v(t,\cdot),v_t(t,\cdot)\bigr)\|_{L^m}
    & \lesssim (1+t)^{-n(1-\frac1m)} \, \|(v_0,v_1)\|_{\mathcal{D}_m^1} \,,
\intertext{and the $L^m-L^m$ estimates}
\label{eq:vm}
\|v(t,\cdot)\|_{L^m}
    & \lesssim (1+t) \, \|(v_0,v_1)\|_{L^m}\,, \\
\label{eq:gradvm}
\|\bigl(\nabla v(t,\cdot),v_t(t,\cdot)\bigr)\|_{L^m}
    & \lesssim \|(v_0,v_1)\|_{H^{1,m}\times L^m} \,,
\end{align}
\end{Thm}
\begin{Rem} \label{Rem61}
In the special case~$m=2$ one can directly prove Theorem~\ref{Thm:halflin} by using the approach presented in Sections
\ref{sec:Shi}-\ref{sec:par}-\ref{sec:hyp}. On the other hand, if~$m\in(1,2)$ then we need different tools.
\end{Rem}
\begin{proof}
Since $v_0,v_1\in L^1$ we can perform Fourier transform of~\eqref{eq:lin} for~$\sigma=1/2$ obtaining the following Cauchy problem
for~$w(t,\xi)=\widehat{v}(t,\xi)$:
\begin{equation}\label{eq:w}
\begin{cases}
w_{tt}+ \mu\,|\xi| w_t + |\xi|^2 w = 0,\\
w(0,\xi) = \widehat{v_0}(\xi),\\
w_t(0,\xi) = \widehat{v_1}(\xi).
\end{cases}
\end{equation}
{}First let~$\mu=2$. In such a case the characteristic root of the symbol of the operator from~\eqref{eq:w} is~$|\xi|$ with multiplicity~$2$. This gives
the representation
\[ w(t,\xi) = \Big(C_1(\xi) + C_2(\xi) t \Big) \,e^{-t|\xi|}. \]
{}From the initial data we immediately get~$C_1=\widehat{v_0}(\xi)$ and since
\[ w_t(t,\xi) = \Big( -\widehat{v_0}(\xi)|\xi| + C_2(\xi) (1-t|\xi|) \Big) \,e^{-t|\xi|}, \]
we get~$C_2=\widehat{v_0}(\xi)|\xi|+\widehat{v_1}(\xi)$, that is,
\begin{align}
\label{eq:w1} w(t,\xi)
    & = \Big( (1+t|\xi|)\,\widehat{v_0}(\xi) + t\,\widehat{v_1}(\xi) \Big) e^{-t|\xi|},\\
\label{eq:w'} w_t(t,\xi)
    & = \Big( -t|\xi|^2 \,\widehat{v_0}(\xi) + (1-t|\xi|)\,\widehat{v_1}(\xi) \Big) e^{-t|\xi|}.
\end{align}
To derive $L^m-L^m$ and $L^1-L^m$ estimates we use tools from the paper~\cite{NR}. Due to the relation
\[ \int_{\R^n} e^{-2\pi\,|\xi|t} \, e^{-2\pi\,i x\cdot \xi}\,d\xi = c_n\,\frac{t}{(t^2+|\xi|^2)^{\frac{n+1}2}}, \]
and by virtue of Young's inequality we conclude
\begin{align}
\nonumber
\|v(t,\cdot)\|_{L^m}
    & \lesssim \|v_0\|_{L^m} + t \|v_1\|_{L^m} \\
\label{eq:vsmall}
    & \lesssim (1+t) \|(v_0,v_1)\|_{L^m},\\
\nonumber
\|v(t,\cdot)\|_{L^m}
    & \lesssim t^{-n(1-1/m)} \left( \|v_0\|_{L^1} + t \|v_1\|_{L^1} \right) \\
\label{eq:vlarge}
    & \lesssim (1+t)\, t^{-n(1-1/m)} \|(v_0,v_1)\|_{L^1},
\end{align}
for the solution~$v$ to~\eqref{eq:lin} and
\begin{align}
\label{eq:gradvsmall}
\|\bigl(\nabla v(t,\cdot),v_t(t,\cdot)\bigr)\|_{L^m}
    & \lesssim \|(\nabla v_0,v_1)\|_{L^m},\\
\label{eq:gradvlarge}
\|\bigl(\nabla v(t,\cdot),v_t(t,\cdot)\bigr)\|_{L^m}
    & \lesssim t^{-n(1-1/m)} \|(v_0,v_1)\|_{L^1},
\end{align}
for its gradient and its time derivative. By using~\eqref{eq:vsmall} (resp.~\eqref{eq:gradvsmall}) for~$t\leq1$ and~\eqref{eq:vlarge} (resp.~\eqref{eq:gradvlarge})
for~$t\geq1$ we derive~\eqref{eq:v} (resp.~\eqref{eq:gradv}). On the other hand, \eqref{eq:vsmall} and \eqref{eq:gradvsmall} directly give~\eqref{eq:vm}
and~\eqref{eq:gradvm}.
\\
If~$\mu\neq2$, then we have two different characteristic roots:
\[ \lambda_\pm = \begin{cases}
\bigl(-\mu \pm \sqrt{\mu^2-4}\bigr)\, |\xi| /2 & \text{if~$\mu>2$, real-valued roots,}\\
\bigl(-\mu \pm i \sqrt{4-\mu^2}\bigr)\, |\xi| /2 & \text{if~$\mu<2$, complex-valued roots.}
\end{cases} \]
Nevertheless, following~\cite{NR} one can prove again the estimates \eqref{eq:vsmall}-\eqref{eq:vlarge}-\eqref{eq:gradvsmall}-\eqref{eq:gradvlarge} and
conclude \eqref{eq:v}-\eqref{eq:gradv}-\eqref{eq:vm}-\eqref{eq:gradvm}.
\end{proof}

\subsection{The case~$\sigma=1$}\label{sec:Shi}
This case was studied in detail in \cite{Shibata}. Here and in Sections~\ref{sec:par} and~\ref{sec:hyp} we deal with $L^1\cap L^2$
estimates, whereas in Theorem~\ref{Thm:halflin} we stated $L^1\cap L^m$ estimates for any~$m\in(1,2]$. In facts, the choice~$m=2$ allows us to use
Parseval's formula in the proofs of Theorems \ref{Thm:Shilin}-\ref{Thm:parlin}-\ref{Thm:hyplin}.
\begin{Thm}\label{Thm:Shilin}
Let~$\sigma=1$ in~\eqref{eq:lin}. Let $n\geq2$ and let~$(v_0,v_1)\in\mathcal{D}_2^2$. Then the solution to~\eqref{eq:lin}, its first derivative in time, and its derivatives in space up to the second order, satisfy the $(L^1\cap L^2)-L^2$ estimates
\begin{align}
\label{eq:vShi}
\|v(t,\cdot)\|_{L^2}
    & \lesssim \begin{cases}
        (1+t)^{-\frac{n-2}4} \, \|(v_0,v_1)\|_{L^1\cap L^2} & \text{if~$n\geq3$,}\\
        \log (e+t) \, \|(v_0,v_1)\|_{L^1\cap L^2} & \text{if~$n=2$,}
        \end{cases} \\
\label{eq:vtShi} \|v_t(t,\cdot)\|_{L^2}
    & \lesssim (1+t)^{-\frac{n}4} \, \|(v_0,v_1)\|_{L^1\cap L^2},\\
\label{eq:vxShi}
\|\nabla v(t,\cdot)\|_{L^2}
    & \lesssim (1+t)^{-\frac{n}4} \|(v_0,v_1)\|_{\mathcal{D}_2^1},\\
\label{eq:vxxShi}
\|\nabla^2 v(t,\cdot)\|_{L^2}
    & \lesssim (1+t)^{-\frac{n+2}4} \|(v_0,v_1)\|_{\mathcal{D}_2^2},
\intertext{and the $L^2-L^2$ estimates}
\label{eq:vtShi2} \|v_t(t,\cdot)\|_{L^2}
    & \lesssim \|(v_0,v_1)\|_{L^2},\\
\label{eq:vxShi2}
\|\nabla v(t,\cdot)\|_{L^2}
    & \lesssim \|(v_0,v_1)\|_{H^1\times L^2},\\
\label{eq:vxxShi2}
\|\nabla^2 v(t,\cdot)\|_{L^2}
    & \lesssim (1+t)^{-\frac12} \|(v_0,v_1)\|_{H^2\times L^2}.
\end{align}
\end{Thm}
\begin{Rem}\label{Rem:NR22}
One might expect the $L^2-L^2$ estimate
\begin{equation}
\label{eq:vShi2false} \|v(t,\cdot)\|_{L^2} \lesssim (1+t)^{\frac12} \, \|(v_0,v_1)\|_{L^2}
\end{equation}
for the solution~$v$ to~\eqref{eq:lin}. But, unfortunately, we are not able to prove~\eqref{eq:vShi2false} for~$n\geq2$. We refer to Proposition 11
in~\cite{NR}, which one can use to prove that
\begin{equation}
\label{eq:NR22} \|v(t,\cdot)\|_{L^2} \lesssim (1+t)^{[n/2]\left(1-\frac1{2\sigma}\right)} \, \left( \|v_0\|_{L^2} + t\, \|v_1\|_{L^2} \right)
\end{equation}
for any~$\sigma\in[1/2,1]$.
\end{Rem}
\begin{proof}
We denote by~$E_0(t,x)(v)$ and~$E_\infty(t,x)(v)$ the solution to~\eqref{eq:lin} localized to low and high frequencies, that is,
\begin{equation}\label{eq:Eloc}
E_0(t,x)(v) = \mathcal{F}^{-1} \left(\chi(\xi) \widehat{v}(t,\xi) \right), \quad E_\infty(t,x)(v) = \mathcal{F}^{-1} \left((1-\chi(\xi)) \widehat{v}(t,\xi)
\right),
\end{equation}
where~$\chi$ is a smooth function such that~$0\leq \chi(\xi)\leq1$ and $\chi(\xi)=1$ for~$|\xi|\leq1$, $\chi(\xi)=0$ for~$|\xi|\geq3/2$. Then the following
estimates can be concluded from \cite{Shibata}:
\begin{align}
\label{eq:Shilowsmall} \|\partial_t^j\partial_x^\alpha E_0(t,\cdot)(v)\|_{L^m}
    & \leq C_{p,m,j,\alpha} \|(v_0,v_1)\|_{L^p}
\intertext{for any~$t\leq2$ and~$1\leq p\leq m\leq\infty$ (Theorem 2.1 (1) from \cite{Shibata}), and}
\label{eq:Shihighgen} \|\partial_t^j\partial_x^\alpha E_\infty(t,\cdot)(v)\|_{L^m}
    & \leq C_{m,j,\alpha,N} e^{-ct} \left( t^{-N/2} \|(v_0,v_1)\|_{H^{[2j+|\alpha|-2-N]^+,m}} + \|v_0\|_{H^{|\alpha|,m}} + \|v_1\|_{H^{[|\alpha|-2]^+,m}} \right)
\intertext{for any~$t>0$, for any $N\in\N$ (also~$N=0$ is allowed) and for~$1<m<\infty$ (Theorem 2.2 (1) from \cite{Shibata}). Moreover, if we fix the
regularity~$L^1\cap L^2$ for the data, then one can prove that}
\label{eq:Shilowlarge} \|\partial_t^j\partial_x^\alpha E_0(t,\cdot)(v)\|_{L^2}
    & \leq \begin{cases}
        C_{0,0} \bigl((1+t)^{-\frac12}\|v_0\|_{L^1} + \log(2+t) \,\|v_1\|_{L^1}\bigr) & \text{if~$n=2$, $j=|\alpha|=0$,}\\
        C_{j,\alpha} (1+t)^{-\frac{n-2}4-\frac{j+|\alpha|}2} \bigl( (1+t)^{-\frac12}\|v_0\|_{L^1} + \|v_1\|_{L^1}\bigr) & \text{otherwise,}
        \end{cases}
\end{align}
for any~$t\geq2$ (Theorem 2.1 (3) from \cite{Shibata}).
\\
For the sake of simplicity, let~$n\geq3$ or~$j=|\alpha|=0$ in what follows, being this special case completely analogous. Combining~\eqref{eq:Shilowsmall}
for~$m=2$ and~$p=1$ with~\eqref{eq:Shilowlarge} we get
\begin{equation}
\label{eq:Shilow} \|\partial_t^j\partial_x^\alpha E_0(t,\cdot)(v)\|_{L^2} \leq C_{j,\alpha} (1+t)^{-\frac{n-2}4-\frac{j+|\alpha|}2} \left(
(1+t)^{-\frac12}\|v_0\|_{L^1} + \|v_1\|_{L^1}\right).
\end{equation}
Now let~$N=0$ and let either~$(j,|\alpha|)=(1,0)$ or $j=0$ and~$|\alpha|=0,1,2$, in~\eqref{eq:Shihighgen}. In such a case,
the term~$t^{-N/2} \|(v_0,v_1)\|_{H^{[2j+|\alpha|-2-N]^+,m}}$ is controlled by the other two in parentheses, that is,
\begin{equation}
\label{eq:Shihigh} \|\partial_t^j\partial_x^\alpha E_\infty(t,\cdot)(v)\|_{L^m}\leq C_{m,j,\alpha,0} e^{-ct} \left( \|v_0\|_{H^{|\alpha|,m}} +
\|v_1\|_{H^{[|\alpha|-2]^+,m}} \right).
\end{equation}
Therefore, from~\eqref{eq:Shilow} and~\eqref{eq:Shihigh} with~$m=2$ we obtain
\begin{equation}
\label{eq:Shi} \|\partial_t^j\partial_x^\alpha v(t,\cdot)\|_{L^2} \lesssim (1+t)^{-\frac{n-2}4-\frac{j+|\alpha|}2} \left( (1+t)^{-\frac12}\|v_0\|_{L^1\cap
H^{|\alpha|}} + \|v_1\|_{L^1\cap L^2}\right)
\end{equation}
for any~$t\geq0$ if either~$(j,|\alpha|)=(1,0)$ or $j=0$ and~$|\alpha|=0,1,2$. This concludes the proof of
\eqref{eq:vShi}-\eqref{eq:vtShi}-\eqref{eq:vxShi}-\eqref{eq:vxxShi}.
\\
To prove \eqref{eq:vtShi2}-\eqref{eq:vxShi2}-\eqref{eq:vxxShi2} it is sufficient to combine the $L^2-L^2$ estimates for~$\partial_t^j\partial_x^\alpha
E_0(t,x)(v)$ as they appear in Theorem 2.1, estimate (5) of~\cite{Shibata} with~\eqref{eq:Shihigh}.
\end{proof}


\subsection{The case~$\sigma\in(0,1/2)$}\label{sec:par}

In this case we want to prove the following statement.
\begin{Thm} \label{Thm:parlin}
Let~$\sigma\in(0,1/2)$ in~\eqref{eq:lin}. Let $n \geq 2$ and let~$(v_0,v_1)\in\mathcal{D}_2^1$. Then the solution to~\eqref{eq:lin} and its first derivatives
with respect to time and space satisfy the $(L^1\cap L^2)-L^2$ estimates
\begin{align}
\label{eq:vpar}
\|v(t,\cdot)\|_{L^2}
    & \lesssim (1+t)^{-\left(\frac{n}{4}-\sigma\right)\,\frac{1}{1-\sigma}} \|(v_0,v_1)\|_{L^1\cap L^2}, \\
\label{eq:vtpar}
\|v_t(t,\cdot)\|_{L^2}
    & \lesssim (1+t)^{-\left(\frac{n}{4}-\sigma\right)\,\frac{1}{1-\sigma}-1} \, \|(v_0,v_1)\|_{\mathcal{D}_2^1},\\
\label{eq:vxpar}
\|\nabla v(t,\cdot)\|_{L^2}
    & \lesssim (1+t)^{-\left(\frac{n+2}{4}-\sigma\right)\,\frac{1}{1-\sigma}} \, \|(v_0,v_1)\|_{\mathcal{D}_2^1},
\intertext{and the $L^2-L^2$ estimates}
\label{eq:vpar2}
\|v(t,\cdot)\|_{L^2}
    & \lesssim (1+t)
     \|(v_0,v_1)\|_{L^2}, \\
\label{eq:vtpar2}
\|v_t(t,\cdot)\|_{L^2}
    & \lesssim  \, \|(v_0,v_1)\|_{H^1\times L^2},\\
\label{eq:vxpar2}
\|\nabla v(t,\cdot)\|_{L^2}
    & \lesssim (1+t)^{-\frac12\,\frac{1-2\sigma}{1-\sigma}} \, \|(v_0,v_1)\|_{H^1\times L^2}.
\end{align}
\end{Thm}
\begin{proof}
We claim that
\begin{equation}\label{eq:D012}
\|\partial_t^j\partial_x^\alpha u(t,\cdot)\|_{L^2} \lesssim (1+t)^{-\left(\frac{n}4+\frac{|\alpha|}2\right)\, \frac1{1-\sigma}-j} \|u_0\|_{L^1\cap
H^{|\alpha|}} + (1+t)^{-\left(\frac{n}4+\frac{|\alpha|}2-\sigma\right)\,\frac1{1-\sigma}-j} \|u_1\|_{L^1\cap L^2}
\end{equation}
for any~$t>0$ and~$j+|\alpha|=0,1$, and even for any~$\sigma\in[0,1/2)$. Once we have proved~\eqref{eq:D012} our estimates
\eqref{eq:vpar}-\eqref{eq:vtpar}-\eqref{eq:vxpar} follow immediately. We can write the solution to~\eqref{eq:lin} as
\begin{equation}
\label{eq:vK} v(t,x) = K_0(t,x) \ast_{(x)} v_0(x) + K_1(t,x)\ast_{(x)} v_1(x),
\end{equation}
where
\begin{equation}
\label{eq:vF} \widehat{K_0} (t,\xi) = \frac{\lambda_+e^{\lambda_-\,t}-\lambda_-e^{\lambda_+\,t}}{\lambda_+-\lambda_-}, \quad \widehat{K_1} (t,\xi) =
\frac{e^{\lambda_-\,t}-e^{\lambda_+\,t}}{\lambda_+-\lambda_-}.
\end{equation}
The characteristic roots~$\lambda_\pm(\xi)$ have non-positive real parts and they are given by
\begin{equation}\label{eq:lambdapm}
\lambda_\pm = \begin{cases}
\bigl(-\mu \pm \sqrt{\mu^2-4|\xi|^{2(1-2\sigma)}}\bigr)\, |\xi|^{2\sigma}/2 & \text{if~$|\xi|^{1-2\sigma} \leq \mu/2$, the roots are real-valued,}\\
\bigl(-\mu \pm i\,\sqrt{4|\xi|^{2(1-2\sigma)}-\mu^2}\bigr)\, |\xi|^{2\sigma}/2 & \text{if~$|\xi|^{1-2\sigma} \geq \mu/2$, the roots are complex-valued.}
\end{cases}
\end{equation}
As in the proof of Theorem~\ref{Thm:Shilin} we denote by~$E_0(t,x)(v)$ and~$E_\infty(t,x)(v)$ the solution to~\eqref{eq:lin} localized to low and high
frequencies. We notice that
\begin{equation}\label{eq:lambdalowpar}
\lambda_+\approx \lambda_+^{(l)} \doteq -|\xi|^{2(1-\sigma)}, \quad \lambda_-\approx \lambda_-^{(l)} \doteq -|\xi|^{2\sigma}, \quad
\lambda_+-\lambda_-\approx \delta^{(l)} \doteq |\xi|^{2\sigma},
\end{equation}
for low frequencies~$|\xi| \leq \varepsilon$, whereas
\begin{equation}\label{eq:lambdahighpar}
\lambda_\pm\approx \lambda_\pm^{(h)} \doteq -|\xi|^{2\sigma}\pm i|\xi|, \quad \lambda_+-\lambda_-\approx \delta^{(h)} \doteq i|\xi|,
\end{equation}
for high frequencies~$|\xi| \geq \frac{1}{\varepsilon}$, where $\varepsilon$ is sufficiently small.

\noindent We shall estimate the~$L^2$ norm of~$E_0(t,x)(v)$, $\partial_t E_0(t,x)(v)$ and $\nabla E_0(t,x)(v)$ by the~$L^1$ norm of $(v_0,v_1)$. Due to
Young's inequality we have to estimate~$\|\chi(D)\partial_t^j\partial_x^\alpha K_0(t,\cdot)\|_{L^2}, \,\, \|\chi(D)\partial_t^j\partial_x^\alpha K_1(t,\cdot)\|_{L^2}$
(see~\eqref{eq:vK}) for~$j+|\alpha|=0,1$. Here $\chi=\chi(\xi)$ is a smooth decreasing function with
$\chi(\xi)=1$ for $|\xi|\leq \varepsilon/2$ and $\chi(\xi)=0$ for $|\xi|\geq \varepsilon$. Due to Parseval's formula we have to estimate
\begin{align}
\label{eq:I0} I_0^2(j,|\alpha|)
    & \doteq \int_{\R^n} \frac{\bigl|\partial_t^j\bigl(\lambda_+ e^{\lambda_- t}-\lambda_- e^{\lambda_+ t}\bigr)\bigr|^2}
    {|\lambda_+ - \lambda_-|^2} |\xi|^{2|\alpha|} \, \chi(\xi)^2 d\xi,\\
\label{eq:I1} I_1^2(j,|\alpha|)
    & \doteq \int_{\R^n} \frac{\bigl|\partial_t^j\bigl(e^{\lambda_+  t}-e^{\lambda_- t}\bigr)\bigr|^2}
    {|\lambda_+ - \lambda_-|^2} |\xi|^{2|\alpha|}
    \, \chi(\xi)^2 d\xi.
\end{align}
We only estimate~$I_1(j,|\alpha|)$. These integrals imply the decay in the estimates
\eqref{eq:vpar}-\eqref{eq:vtpar}-\eqref{eq:vxpar}. In same way we estimate~$I_0(j,|\alpha|)$. We get for $j=0$
\[ \frac{\bigl|e^{\lambda_+  t}-e^{\lambda_- t}\bigr|^2}
    {|\lambda_+ - \lambda_-|^2}\chi(\xi)^2 \approx \frac{\bigl|e^{\lambda_+^{(l)} t}-e^{\lambda_-^{(l)}t}\bigr|^2}{{\delta^{(l)}}^2} \approx \frac{e^{-2|\xi|^{2(1-\sigma)}t}}{|\xi|^{4\sigma}}, \]
and for $j=1$
\[ \frac{\bigl|\partial_t\bigl(e^{\lambda_+  t}-e^{\lambda_- t}\bigr)\bigr|^2}
    {|\lambda_+ - \lambda_-|^2}\chi(\xi)^2 \approx
     \frac{\bigl|\lambda_+^{(l)}
e^{\lambda_+^{(l)} t}-\lambda_-^{(l)}e^{\lambda_-^{(l)}t}\bigr|^2}{{\delta^{(l)}}^2} \lesssim \frac{|\xi|^{4(1-\sigma)} e^{-2|\xi|^{2(1-\sigma)}t} + |\xi|^{4 \sigma} e^{-2|\xi|^{2\sigma}t}}{|\xi|^{4\sigma}}. \]
By the change of variables~$\eta=\xi\,t^{\frac1{2(1-\sigma)}}$ we get
\begin{eqnarray*} && I_1^2(0,|\alpha|) \lesssim \int_{|\xi|\leq \varepsilon} |\xi|^{-4\sigma +2|\alpha|+n-1} e^{-2|\xi|^{2(1-\sigma)}t} \,d|\xi| \lesssim t^{\frac{4\sigma-n-2|\alpha|}{2(1-\sigma)}} \int_0^\infty
|\eta|^{-4\sigma +2|\alpha|+ n-1} e^{-2|\eta|} \,d\eta \\ && \qquad \lesssim t^{\frac{4\sigma-n-2|\alpha|}{2(1-\sigma)}}\,\,\,\mbox{for large}\,\,\,t. \end{eqnarray*}
We remark that~$-4\sigma +2|\alpha|+n > 0$ for any~$n\geq2$. The same reasoning gives
\begin{eqnarray*} && I_1^2(1,|\alpha|) \lesssim \int_{|\xi|\leq \varepsilon} |\xi|^{n-1+2|\alpha|} \bigl(|\xi|^{4(1-2\sigma)} e^{-2|\xi|^{2(1-\sigma)}t} + e^{-2|\xi|^{2\sigma}t}\bigr) \,d|\xi| \\ && \qquad \lesssim t^{-\frac{4(1-2\sigma)+n+2|\alpha|}{2(1-\sigma)}} \int_0^\infty
|\eta|^{4(1-2\sigma) +2|\alpha|+ n-1} e^{-2|\eta|} \,d\eta + t^{\frac{-n-2|\alpha|}{2\sigma}} \int_0^\infty
|\eta|^{2|\alpha|+ n-1} e^{-2|\eta|} \,d\eta\\ && \qquad \lesssim t^{\frac{-4(1-2\sigma)-n-2|\alpha|}{2(1-\sigma)}}
\,\,\,\mbox{for large}\,\,\,t.\end{eqnarray*}
In the last step we used for $n \geq 2$ the inequality
\[    \frac{4(1-2\sigma) +2|\alpha|+ n}{2(1-\sigma)} \leq \frac{n+2|\alpha|}{2\sigma}.\]
Now we have to estimate the $L^2$ norm of~$\partial_t^j\partial_x^\alpha E_\infty(t,x)(v)$ for~$j+|\alpha|=0,1$, which is equivalent to the~$L^2$ norm of $|\xi|^\alpha\partial_t^j F_{x \to \xi}\big(E_\infty(t,x)(v)\big)$. It is sufficient to estimate for large frequencies $|\xi|\geq \frac{1}{\varepsilon}$
\begin{align*}
|\widehat{K_0} (t,\xi)| & \lesssim e^{-|\xi|^{2\sigma}t}, \\
\langle \xi \rangle |\widehat{K_1} (t,\xi)| & \lesssim  \, e^{-|\xi|^{2\sigma}t}, \\
\langle \xi \rangle^{-1} |\partial_t \widehat{K_0} (t,\xi)| & \lesssim \langle \xi \rangle^{-1} \, \bigl(\langle \xi \rangle e^{-|\xi|^{2\sigma}t}\bigr), \\
|\partial_t \widehat{K_1} (t,\xi)| & \lesssim e^{-|\xi|^{2\sigma}t}.
\end{align*}
 All terms are controlled by~$e^{-C^{2\sigma}t}$ for~$|\xi|\geq \frac{1}{\varepsilon}$ (uniformly with respect to~$t>0$). Indeed, due to Parseval's formula the $L^2$ norm of $\widehat{v_0}, \widehat{v_1}, |\xi|\widehat{v_0}$ are equivalent to the $L^2$ norm of $v_0, v_1, \nabla v_0$.
\\
We remark that the exponential decay~$e^{-C^{2\sigma}t}$ for the high frequencies is better than the potential decay for the low frequencies. The middle zone $\{|\xi|\in [\varepsilon,\frac{1}{\varepsilon}]\}$ brings an exponential decay ~$e^{-C^{2\sigma}t}$, too, if we recall that the real part of the characteristic roots $\lambda_{\pm}$ is negative there. This concludes the proof of~\eqref{eq:D012}.

To prove~\eqref{eq:vpar2}-\eqref{eq:vtpar2}-\eqref{eq:vxpar2} it is sufficient to estimate the $L^2$ norm of $|\xi|^\alpha\partial_t^j F_{x \to \xi}\big(E_0(t,x)(v)\big)$ for small frequencies and for~$j+|\alpha|=0,1$.
By using the presented approach one can directly derive~\eqref{eq:vxpar2}.
To prove~\eqref{eq:vpar2} and~\eqref{eq:vtpar2} we have to estimate the $L^\infty$ norm of $\partial_t^j\widehat{K_i}(t,\xi)$ for~$i,j=0,1$. Here the estimates for $\partial_t^j\widehat{K_1}(t,\xi)$
are of interest. Using
\[ \widehat{K_1}(t,\xi)=t e^{\lambda_+ t}\int_0^1 e^{-t\theta\sqrt{|\xi|^{4\sigma}-4|\xi|^2}\,d\theta} \]
we obtain
\[ \|\widehat{K_1}(t,\cdot)\|_{L^\infty(\R^n_\xi)} \lesssim t,\quad \|\partial_t\widehat{K_1}(t,\cdot)\|_{L^\infty(\R^n_\xi)} \lesssim 1.  \]
This completes the proof.
\end{proof}


\subsection{The case~$\sigma\in(1/2,1)$}\label{sec:hyp}

In this case we want to prove the following statement.
\begin{Thm} \label{Thm:hyplin}
Let~$\sigma\in(1/2,1)$ in~\eqref{eq:lin}. Let $n\geq2$ and let~$(v_0,v_1)\in\mathcal{D}_2^1$. Then the solution to~\eqref{eq:lin}, its first derivatives in time and space, and its derivative in space of fractional order~$2\sigma$ satisfy the $(L^1\cap L^2)-L^2$ estimates
\begin{align}
\label{eq:vhyp}
\|v(t,\cdot)\|_{L^2}
    & \lesssim \begin{cases}
    (1+t)^{-\frac{n-2}{4\sigma}} \|(v_0,v_1)\|_{L^1\cap L^2}, & \text{if~$n\geq3$,}\\
    \log (e+t) \|(v_0,v_1)\|_{L^1\cap L^2}, & \text{if~$n=2$,}
    \end{cases} \\
\label{eq:vthyp} \|v_t(t,\cdot)\|_{L^2}
    & \lesssim (1+t)^{-\frac{n}{4\sigma}} \, \|(v_0,v_1)\|_{\mathcal{D}_2^{2(1-\sigma)}},\\
\label{eq:vxhyp}
\|\nabla v(t,\cdot)\|_{L^2}
    & \lesssim (1+t)^{-\frac{n}{4\sigma}} \|(v_0,v_1)\|_{\mathcal{D}_2^1},\\
\label{eq:vxxhyp}
\|v(t,\cdot)\|_{\dot{H}^{2\sigma}}
    & \lesssim (1+t)^{-\frac{n-2}{4\sigma}-1} \|(v_0,v_1)\|_{\mathcal{D}_2^{2\sigma}},
\intertext{and the $L^2-L^2$ estimate}
\label{eq:vhyp2}
\|v(t,\cdot)\|_{L^2}
    & \lesssim \, (1+t)\|(v_0,v_1)\|_{L^2},\\
\label{eq:vthyp2} \|v_t(t,\cdot)\|_{L^2}
    & \lesssim \, \|(v_0,v_1)\|_{H^{2(1-\sigma)}\times L^2},\\
\label{eq:vxhyp2}
\|\nabla v(t,\cdot)\|_{L^2}
    & \lesssim \|(v_0,v_1)\|_{H^1\times L^2},\\
\label{eq:vxxhyp2}
\|v(t,\cdot)\|_{\dot{H}^{2\sigma}}
    & \lesssim \|(v_0,v_1)\|_{H^{2\sigma}\times L^2}.
\end{align}
\end{Thm}
\begin{proof}
We claim that
\begin{equation}
\label{eq:Shi121}
\|\partial_t^j v(t,\cdot)\|_{\dot{H}^\kappa}\leq C_{j,s} (1+t)^{-\frac{n-2}{4\sigma}-\frac{j+\kappa}{2\sigma}} \left( (1+t)^{-\frac1{2\sigma}}\|v_0\|_{L^1\cap H^{2j(1-\sigma)+\kappa}} + \|v_1\|_{L^1\cap L^2}\right)
\end{equation}
for either~$(j,\kappa)=(1,0)$, or $j=0$ and~$\kappa=0,1,2\sigma$, an exception is given for the case~$n=2$ and $j=\kappa=0$. In the following we
put~$\dot{H}^0=H^0=L^2$. Analogously to the proof of Theorem~\ref{Thm:parlin} the characteristic roots~$\lambda_\pm(\xi)$ are given by~\eqref{eq:lambdapm},
but now we have~\eqref{eq:lambdahighpar} for low frequencies~$|\xi| \leq \varepsilon$ and~\eqref{eq:lambdalowpar} for high frequencies~$|\xi| \geq
\frac{1}{\varepsilon}$ (since now~$1-2\sigma<0$ in~\eqref{eq:lambdapm}). That is, formulas for~$\lambda_\pm^{(l)},\delta^{(l)}$
and~$\lambda_\pm^{(h)},\delta^{(h)}$ are exchanged. The middle frequencies $\{|\xi|\in [\varepsilon,\frac{1}{\varepsilon}]\}$ are considered as in
Theorem~\ref{Thm:parlin}.
\\
Again we denote by~$E_0(t,x)(v)$ and~$E_\infty(t,x)(v)$ the solution to~\eqref{eq:lin} localized to low and high frequencies. First we estimate the $L^2$ norms of $|\xi|^\kappa\partial_t\widehat{K_i}$ for~$i=0,1$. Again we introduce $I_0(j,\kappa)$ and~$I_1(j,\kappa)$ as in \eqref{eq:I0}-\eqref{eq:I1} and the essential estimates appear from~$I_1(j,\kappa)$. By the change of variables~$\eta=\xi\,t^{\frac1{2\sigma}}$ we get for small frequencies
\begin{eqnarray*} &&  I_1^2(j,\kappa) \lesssim \int_{|\xi|\leq \varepsilon} |\xi|^{2(\kappa+j-1)} e^{-2|\xi|^{2\sigma}t} \,d\xi \lesssim t^{\frac{2(1-\kappa-j)-n}{2\sigma}} \int_{0}^\infty |\eta|^{2(\kappa+j-1)} |\eta|^{n-1} e^{-2|\eta|} \,d|\eta| \\ && \qquad \lesssim  t^{\frac{2(1-\kappa-j)-n}{2\sigma}}. \end{eqnarray*}
We remark that~$2(\kappa+j-1) > -n$ for any~$n\geq3$ and for $n=2$ if~$\kappa+j>0$. If~$n=2$ and $\kappa=j=0$, then we use for small frequencies the
relation
\begin{equation}\label{eq:K1small}
\widehat{K_1}(t,\xi)=e^{-\frac{\mu}{2}|\xi|^{2\sigma}t} t \frac{\sin(\alpha t)}{\alpha t}\,\,\,\mbox{with}\,\,\,
\alpha=\frac{1}{2}|\xi|^{2\sigma}\sqrt{4|\xi|^{2(1-2\sigma)}-\mu^2}.
\end{equation}
For any~$t>0$ let the function~$\rho=\rho(t)$ be defined by
\[ \frac{1}{2}\rho^{2\sigma}\sqrt{4\rho^{2(1-2\sigma)}-\mu^2} = \frac1t. \]
Since~$\sin (\alpha t)\lesssim \alpha t$ for any~$\alpha\leq 1/t$, and~$n=2$, it follows that
\[ I_1^2(0,0) \lesssim \int_{|\xi|\leq\rho(t)} e^{-\mu\,|\xi|^{2\sigma}t} t^2\,d\xi + \int_{|\xi|\geq\rho(t)} |\xi|^{-2} e^{-\mu\,|\xi|^{2\sigma}t}
\,d\xi \lesssim (t\rho(t))^2 + \log (e+\rho(t)^{-1}) \approx \log (e+t). \]
Indeed, $\rho(t)\approx 1/t$. Analogously, we proceed for~$I_0(j,\kappa)$. Then we estimate for large frequencies
\begin{align*}
|\widehat{K_0} (t,\xi)| & \lesssim e^{-|\xi|^{2(1-\sigma)}t}, \\
\langle \xi \rangle^{2\sigma} |\widehat{K_1} (t,\xi)| & \lesssim  \, e^{-|\xi|^{2(1-\sigma)}t}, \\
\langle \xi \rangle^{-2(1-\sigma)} |\partial_t \widehat{K_0} (t,\xi)| & \lesssim \langle \xi \rangle^{-2(1-\sigma)} \, \bigl(\langle \xi \rangle^{2(1-\sigma)} e^{-|\xi|^{2(1-\sigma)}t}\bigr), \\
|\partial_t \widehat{K_1} (t,\xi)| & \lesssim \, e^{-|\xi|^{2(1-\sigma)}t}
\end{align*}
which are all controlled by~$e^{-|\xi|^{2(1-\sigma)}t}$ for~$|\xi|\geq \frac{1}{\varepsilon}$ (uniformly with respect to~$t>0$).
Analogously to the proof of Theorem~\ref{Thm:parlin} this concludes the proof of~\eqref{eq:Shi121}.
For the proof of \eqref{eq:vhyp2} we use relation~\eqref{eq:K1small} for small frequencies. The proof of
\eqref{eq:vthyp2}-\eqref{eq:vxhyp2}-\eqref{eq:vxxhyp2} immediately follows since $|\xi|^\kappa\widehat{K_i}(t,\xi)$ with~$\kappa\geq1$ and
$\partial_t\widehat{K_i}(t,\xi)$ are bounded for small frequencies and~$i=0,1$. Indeed,
\[ |\xi|^\kappa\,\bigl(|\widehat{K_0} (t,\xi)| + |\widehat{K_1} (t,\xi)|\bigr) + |\partial_t \widehat{K_0} (t,\xi)| + |\partial_t \widehat{K_1} (t,\xi)| \lesssim \frac{|\xi| e^{-|\xi|^{2\sigma}\,t}}{|\xi|} \lesssim 1.
\]
This completes the proof. \end{proof}
\begin{Rem} \label{Rem8}
The goal of this section was to prove linear estimates for the solution or some derivatives. We have chosen as an upper
bound $C(t)\|(v_0,v_1)\|$ with suitable norms. It is clear that we can get better estimates by using $C_0(t)\|v_0\|+C_1(t)\|v_1\|$.
The above estimates are sufficient to reach the goals of this paper.
\end{Rem}
%

\section{Treatment of corresponding semi-linear models}\label{sec:semilinear}

In this section we will use the decay estimates for~\eqref{eq:lin} which are obtained in Theorems
\ref{Thm:halflin}-\ref{Thm:Shilin}-\ref{Thm:parlin}-\ref{Thm:hyplin} to prove the corresponding Theorems
\ref{Thm:halfnl}-\ref{Thm:Shinl}-\ref{Thm:parnl}-\ref{Thm:hypnl}.
\\
Our main tools are Duhamel's principle and Gagliardo-Nirenberg inequality. Since we are dealing with semi-linear structural damped waves with constant
coefficients in the linear part the application of Duhamel's principle leads to the following:\\
{\it If we write the solution to~\eqref{eq:lin} with the fundamental solutions $G_0$ and $G_1$ in the form
\[ v (t,x) = G_0(t,x) \ast_{(x)} v_0(x) + G_1(t,x) \ast_{(x)} v_1(x), \]
then the solution to~\eqref{eq:diss} becomes
\[ u(t,x) = G_0(t,x) \ast_{(x)} u_0(x) + G_1(t,x) \ast_{(x)} u_1(x) + \int_0^t G_1(t-s,x) \ast_{(x)} f(u(s,x))\,ds. \] }
Let~$m\in(1,2]$ and let~$A$ be a space with norm~$\|\cdot\|_A$. Let us assume that the solution to~\eqref{eq:lin} satisfies some decay estimates in the
form
\begin{equation}\label{eq:abstractdecay}
\|\partial_x^\alpha u (t,\cdot) \|_{L^m} \lesssim f_{|\alpha|}(t) \|(u_0,u_1)\|_A, \ |\alpha|\leq k; \quad \| \partial_t u (t,\cdot) \|_{L^m} \lesssim g(t)
\|(u_0,u_1)\|_A.
\end{equation}
Here the decay functions~$f_{|\alpha|}(t)$ and~$g(t)$ depend, in general, on~$n$. Let us consider the space
\[ X(t)\doteq \mathcal{C}([0,t],H^{k,m})\cap\mathcal{C}^1([0,t],L^m)  \]
with the norm
\begin{align*}
\|w\|_{X(t)}
    & \doteq \sup_{0\leq \tau\leq t} \Bigl( \sum_{|\alpha|\leq k} f_{|\alpha|}(\tau)^{-1}\|\partial_x^\alpha w(\tau,\cdot)\|_{L^m}
    + g(\tau)^{-1}\|\partial_t w(\tau,\cdot)\|_{L^m} \bigr).
\intertext{For the sake of brevity, we also define}
\|w\|_{X_0(t)}
    & \doteq\sup_{0\leq \tau\leq t} \Big( f_{0}(\tau)^{-1}\|w(\tau,\cdot)\|_{L^m} + f_{k}(\tau)^{-1} \sum_{|\alpha|=k}
    \|\partial_x^\alpha w(\tau,\cdot)\|_{L^m} \Big),
\end{align*}
a norm on the space~$X_0(t) \doteq \mathcal{C}([0,t],H^{k,m})$. We remark that if~$w\in X(t)$, then~$\|w\|_{X(s)}\leq\|w\|_{X(t)}$ for any~$s\leq t$, and~$\|w\|_{X_0(t)}\leq\|w\|_{X(t)}$.
\\
We will prove that for any data~$(u_0,u_1)\in A$ the operator~$N$ which is defined for any~$u\in X(t)$ by
\begin{equation}\label{eq:Nw}
Nu(t,x) = G_0(t,x)\ast_{(x)} u_0(x) + G_1(t,x)\ast_{(x)} u_1(x) + \int_0^t G_1(t-s,x)\ast_{(x)} f(u(s,x))\,ds
\end{equation}
satisfies the estimates
\begin{align}
\label{eq:well}
\|Nu\|_{X(t)}
    & \leq C\,\|(u_0,u_1)\|_A + C\|u\|_{X_0(t)}^p, \\
\label{eq:contraction}
\|Nu-Nv\|_{X(t)}
    & \leq C\|u-v\|_{X_0(t)} \bigl(\|u\|_{X_0(t)}^{p-1}+\|v\|_{X_0(t)}^{p-1}\bigr),
\end{align}
uniformly with respect to~$t\in [0,\infty)$. To prove these estimates one should use~\eqref{eq:abstractdecay}.
\\
By standard arguments (see, for instance,~\cite{DALR}) from~\eqref{eq:well} it follows that~$N$ maps~$X(t)$ into itself for small data. Then estimates
\eqref{eq:well}-\eqref{eq:contraction} lead to the existence of a unique solution to $u=Nu$. In fact, taking the recurrence sequence $u_{-1} = 0,\  u_{j} =
N(u_{j-1})$ for  $j=0,1,2,\cdots $, we apply~\eqref{eq:well} with small $\|(u_0,u_1)\|_A=\epsilon$ and we see inductively that
\begin{equation}\label{eq:recurrence}
\|u_j\|_{X(t)}\le C_1\epsilon,
\end{equation}
where $C_1=2C$ for any $\epsilon \in[0,\epsilon_0]$ with~$\epsilon_0= \epsilon_0(C_1)$ sufficiently small.
\\
Once the uniform estimate~\eqref{eq:recurrence} is checked we use~\eqref{eq:contraction} once more and find
\begin{equation}\label{6.4.27}
\|u_{j+1}-u_{j}\|_{X(t)}\le C\epsilon^{p-1}, \,\,\|u_j-u_{j-1}\|_{X(t)} \leq 2^{-1}\|u_j-u_{j-1}\|_{X(t)}
\end{equation}
for $\epsilon\le\epsilon_0$ sufficiently small. {}From~\eqref{6.4.27} we get inductively $\|u_j-u_{j-1}\|_{X(t)}\le {C}{2^{-j}}$ so that $\{u_j \}$ is a
Cauchy sequence in the Banach space $X(t)$ converging to the unique solution of $Nu=u$. We denote such a solution by~$u$. Since all of the constants are
independent of $t$ we can take $t\to \infty$ and we gain a local and a global existence result simultaneously. Finally, we see that the definition of
$\|u\|_{X(t)}$ is chosen in an appropriate way to obtain the decay estimates~\eqref{eq:abstractdecay} for the solution to the semi-linear problem, too.
\\
Therefore, to prove our results we need only to establish \eqref{eq:well} and~\eqref{eq:contraction}.
\\
During the proof a special role will play different applications of Gagliardo-Nirenberg inequality to control~$L^m$ norms of the non-linear term~$f(u)$ for
$m\in[1,2]$ using~\eqref{eq:disscontr} to estimate $\|f(u)\|_{L^m}\lesssim \|u\|_{L^{mp}}^p$. In particular, in what follows, we will use the estimates
\begin{equation}
\label{eq:GN}
\|u(s,\cdot)\|_{L^{q}}^p \lesssim \|u(s,\cdot)\|_{L^m}^{p(1-\theta_k(q))}\,\|\nabla^k u(s,\cdot)\|_{L^m}^{p\theta_k(q)},
\end{equation}
where~$k=1,2$ and
\[ \theta_k(q) \doteq \frac{n}k\left(\frac1m-\frac1q\right), \qquad m\leq q\leq \frac{nm}{n-k\,m}.\]
In particular, these last inequalities give an interval for admissible $p\in [m, n/(n-k\,m)]$ for the exponent~$p$.
\begin{Rem}\label{Rem:u1L2}
Since we have in mind to use Gagliardo-Nirenberg inequality, as stated in \eqref{eq:GN}, it is clear that we can use linear estimates for~\eqref{eq:lin}
only if we make no assumption on the derivatives of~$v_1$. Such a problem does not appear for~$v_0$ which is not involved in the application of Duhamel's
principle.
\\
It is important to notice that in Theorem~\ref{Thm:Shilin} we obtained a decay estimate for the second derivatives of the solution~$v$ with respect to the
spatial variables by assuming additional regularity on~$v_0$, namely~$v_0\in L^1\cap H^2$ with no need of additional regularity for~$v_1$. Such an effect
does not appear in the case~$\sigma\in[0,1/2]$, that is, the assumption~$v_0\in L^1\cap H^k$ for~$k>1$ brings no benefit with respect to~$v_0\in L^1\cap
H^1$, unless we also assume additional regularity for~$v_1$.
\end{Rem}
\noindent We are now ready to prove our statements from Section~\ref{sec:intro}.
%
%
\begin{proof}[Proof of Theorem~\ref{Thm:halfnl}]
Here the space of the data is~$A=\mathcal{D}_m^1$, whereas
\[ X(t)\doteq \mathcal{C}([0,t],H^{1,m})\cap\mathcal{C}^1([0,t],L^m) \]
with the norm
\[ \|w\|_{X(t)} \doteq \sup_{0\leq \tau\leq t} \Bigl( (1+\tau)^{n\left(1-\frac1m\right)-1} \|w(\tau,\cdot)\|_{L^m} + (1+\tau)^{n\left(1-\frac1m\right)} \|\bigl(\nabla w(\tau,\cdot),\partial_tw(\tau,\cdot)\bigr)\|_{L^m} \Bigr). \]
We first prove~\eqref{eq:well}. We use two different strategies for~$s\in [0,t/2]$ and~$s\in [t/2,t]$ to control the integral term in~\eqref{eq:Nw}. In
particular, we use the $(L^1\cap L^m)-L^m$ estimates \eqref{eq:v}-\eqref{eq:gradv} if~$s\in [0,t/2]$ and we use the $L^m-L^m$ estimates \eqref{eq:vm}-\eqref{eq:gradvm} if~$s\in [t/2,t]$. Therefore we get
\begin{align}
\nonumber
\|Nu(t,\cdot)\|_{L^m}
    & \leq C (1+t)^{1-n(1-\frac1m)}\,\|(u_0,u_1)\|_{L^1\times L^m}\\
\nonumber
    & \quad + C\int_0^{t/2} (1+t-s)^{1-n(1-\frac1m)}\,\|f(u(s,\cdot))\|_{L^1\cap
L^m}ds \\
\label{eq:Nu}
    & \quad + C\int_{t/2}^t (1+t-s)\,\|f(u(s,\cdot))\|_{L^m}ds,\\
\nonumber
\|\bigl(\partial_t Nu(t,\cdot),\nabla Nu(t,\cdot)\bigr)\|_{L^m}
    & \leq C (1+t)^{-n(1-\frac1m)}\,\|(u_0,u_1)\|_{\mathcal{D}_1^m}\\
\nonumber
    & \quad + C\int_0^{t/2} (1+t-s)^{-n(1-\frac1m)}\,\|f(u(s,\cdot))\|_{L^1\cap L^m}ds \\
\label{eq:gradNu}
    & \quad + C\int_{t/2}^t \|f(u(s,\cdot))\|_{L^m}ds.
\end{align}
By using~\eqref{eq:disscontr} we can estimate~$|f(u)|\lesssim |u|^p$, so that
\[ \|f(u(s,\cdot))\|_{L^1\cap L^m}\lesssim \|u(s,\cdot)\|_{L^p}^p+ \|u(s,\cdot)\|_{L^{mp}}^p, \]
and, analogously,
\[ \|f(u(s,\cdot))\|_{L^m}\lesssim \|u(s,\cdot)\|_{L^{mp}}^p. \]
Since $p\in[m,n/(n-m)]$ in Theorem~\ref{Thm:halfnl} we can apply~\eqref{eq:GN} for $q=p$ and $q=mp$ and with~$k=1$. In this way we obtain
\begin{align}
\label{eq:fusGNlowp}
\|f(u(s,\cdot))\|_{L^1\cap L^m}
    & \lesssim \|u\|_{X_0(s)}^p (1+s)^{-p\big(n(1-1/m)-1+\theta_1(p)\big)}= \|u\|_{X_0(s)}^p (1+s)^{-p(n-1)+n}
\intertext{since~$\theta_1(p)<\theta_1(mp)$, whereas}
\label{eq:fusGNlow2p}
\|f(u(s,\cdot))\|_{L^m}
    & \lesssim \|u\|_{X_0(s)}^p (1+s)^{-p\big(n(1-1/m)-1+\theta_1(mp)\big)} = \|u\|_{X_0(s)}^p (1+s)^{-p(n-1)+n/m}.
\end{align}
Summarizing we find
\begin{align}
\nonumber
\|Nu(t,\cdot)\|_{L^m}
    & \leq C (1+t)^{1-n(1-\frac1m)}\,\|(u_0,u_1)\|_{L^1\times L^m}\\
\nonumber
    & \quad + C\,\|u\|_{X_0(t)}^p\,\int_0^{t/2} (1+t-s)^{1-n(1-\frac1m)}\,(1+s)^{-(p\,(n-1)-n)}ds \\
\label{eq:Nu2}
    & \quad + C\,\|u\|_{X_0(t)}^p\,\int_{t/2}^t (1+t-s)\,(1+s)^{-(p\,(n-1)-n/m)}ds,\\
\nonumber
\|\bigl(\partial_t Nu(t,\cdot),\nabla Nu(t,\cdot)\bigr)\|_{L^m}
    & \leq C (1+t)^{-n(1-\frac1m)}\,\|(u_0,u_1)\|_{\mathcal{D}_m^1}\\
\nonumber
    & \quad + C\,\|u\|_{X_0(t)}^p\,\int_0^{t/2} (1+t-s)^{-n(1-\frac1m)}\,(1+s)^{-(p\,(n-1)-n)}ds \\
\label{eq:gradNu2}
    & \quad + C\,\|u\|_{X_0(t)}^p\,\int_{t/2}^t (1+s)^{-(p\,(n-1)-n/m)}ds.
\end{align}
The key tool relies now in the estimate
\begin{equation}\label{eq:1ts}
(1+t-s)\approx (1+t), \quad \text{for any~$s\in [0,t/2]$ and} \quad (1+s)\approx (1+t) \quad \text{for any~$s\in [t/2,t]$.}
\end{equation}
Since~$p(n-1)-n >1$ thanks to~\eqref{eq:phalf} it holds
\begin{align*}
\int_0^{t/2} (1+t-s)^{\ell-n(1-\frac1m)}\,(1+s)^{-(p\,(n-1)-n)}ds
    & \approx (1+t)^{\ell-n(1-\frac1m)} \int_0^{t/2} (1+s)^{-(p\,(n-1)-n)}ds \\
    & \approx (1+t)^{\ell-n(1-\frac1m)}, \\
\int_{t/2}^t (1+t-s)^\ell\,(1+s)^{-(p\,(n-1)-n/m)}ds
    & \approx (1+t)^{-(p\,(n-1)-n/m)} \int_0^{t/2} (1+s)^\ell ds \\
    & \approx (1+t)^{-(p\,(n-1)-n/m)+\ell+1} \leq (1+t)^{\ell-n(1-\frac1m)}
\end{align*}
for~$l=0,1$, and this concludes the proof of~\eqref{eq:well}.
\\
Now we prove~\eqref{eq:contraction}. We remark that
\[ \|Nu-Nv\|_{X(t)} = \Big\| \int_0^t G_1(t-s,x)\ast_{(x)} (f(u(s,x))-f(v(s,x))) \,ds \Big\|_{X(t)}. \]
Using again~\eqref{eq:v} if~$s\in [0,t/2]$ and~\eqref{eq:vm} if~$s\in [t/2,t]$ we can estimate
\begin{multline*}
\| G_1(t,s,x) \ast_{(x)} (f(u(s,x))-f(v(s,x))) \|_{L^m} \\
    \lesssim \begin{cases}
(1+t-s)^{1-n(1-1/m)} \,\|f(u(s,\cdot))-f(v(s,\cdot)) \|_{L^1\cap L^m}, & s\in [0,t/2], \\
(1+t-s) \,\|f(u(s,\cdot))-f(v(s,\cdot)) \|_{L^m}, & s\in [t/2,t],
\end{cases}
\end{multline*}
whereas using~\eqref{eq:gradv} if~$s\in [0,t/2]$ and~\eqref{eq:gradvm} if~$s\in [t/2,t]$ we get
\begin{multline*}
\| (\partial_t,\nabla) G_1(t,s,x) \ast_{(x)} (f(u(s,x))-f(v(s,x))) \|_{L^m} \\
    \lesssim \begin{cases}
(1+t-s)^{-n(1-1/m)} \,\|f(u(s,\cdot))-f(v(s,\cdot)) \|_{L^1\cap L^m}, & s\in [0,t/2], \\
\|f(u(s,\cdot))-f(v(s,\cdot)) \|_{L^m}, & s\in [t/2,t].
\end{cases}
\end{multline*}
By using~\eqref{eq:disscontr} and H\"older's inequality we can now estimate
\begin{align*}
\|f(u(s,\cdot))-f(v(s,\cdot)) \|_{L^1}
    & \lesssim \|u(s,\cdot)-v(s,\cdot)\|_{L^p} \, \left(\|u(s,\cdot)\|_{L^p}^{p-1}+ \|v(s,\cdot)\|_{L^p}^{p-1}\right),\\
\|f(u(s,\cdot))-f(v(s,\cdot)) \|_{L^m}
    & \lesssim \|u(s,\cdot)-v(s,\cdot)\|_{L^{mp}} \, \left(\|u(s,\cdot)\|_{L^{mp}}^{p-1}+ \|v(s,\cdot)\|_{L^{mp}}^{p-1}\right).
\end{align*}
Analogously to the proof of~\eqref{eq:well} we apply Gagliardo-Nirenberg inequality to the terms
\[ \|u(s,\cdot)-v(s,\cdot)\|_{L^q}, \qquad \|u(s,\cdot)\|_{L^q}\,, \quad \|v(s,\cdot)\|_{L^q} \]
with~$q=p$ and~$q=mp$, and we conclude the proof of~\eqref{eq:contraction}.
\end{proof}
%
%
\begin{proof}[Proof of Theorem~\ref{Thm:Shinl}]
We follow the proof of Theorem~\ref{Thm:halfnl}. Having in mind  Theorem \ref{Thm:Shilin} we fix now the space~$A=\mathcal{D}_2^2$ for the data, and the
norm
\[ \|u\|_{X(t)} \doteq \sup_{0\leq \tau\leq t} \Bigl( \sum_{k=0}^2 (1+\tau)^{\frac{n-2+2k}4} \|\nabla^k u(\tau,\cdot)\|_{L^2} + (1+\tau)^{\frac{n}4} \|u_t(\tau,\cdot)\|_{L^2}\Bigr) \]
for the space
\[ X(t)=\left\{ u\in \mathcal{C}([0,t],H^2)\cap \mathcal{C}^1([0,t],L^2) \right\}. \]
More precisely, if~$n=2$ then due to Theorem \ref{Thm:Shilin} the coefficient of~$\|u(\tau,\cdot)\|_{L^2}$ is~$(\log(e+\tau))^{-1}$. Since this term brings no additional difficulties we will ignore it. The proof is completely analogous if one replaces such a coefficient in
the definition of~$\|u\|_{X(t)}$.
\\
We only prove~\eqref{eq:well}, being the proof of~\eqref{eq:contraction} analogous, as in the proof of Theorem~\ref{Thm:halfnl}. In order to estimate~$Nu$ we use only~\eqref{eq:uShi} for all~$s\in [0,t]$, that is,
\begin{equation}\label{eq:NuShi}
\|Nu(t,\cdot)\|_{L^2} \leq C (1+t)^{-\frac{n-2}4}\,\|(u_0,u_1)\|_{L^1\times L^2} + C \int_0^t (1+t-s)^{-\frac{n-2}4}\,\|f(u(s,\cdot))\|_{L^1\cap
L^2}ds.
\end{equation}
On the other hand, as in the proof of Theorem~\ref{Thm:halfnl} we use \eqref{eq:vtShi}-\eqref{eq:vxShi}-\eqref{eq:vxxShi} if~$s\in [0,t/2]$ and \eqref{eq:vtShi2}-\eqref{eq:vxShi2}-\eqref{eq:vxxShi2} if~$s\in [t/2,t]$, for estimating the other terms, that is,
we arrive for $(j,|\alpha|) \in \{(1,0),(0,1),(0,2)\}$ at \begin{align}
\| \partial_t^j\partial_x^\alpha Nu(t,\cdot)\|_{L^2}
    & \leq C (1+t)^{-\frac{n+2(j+|\alpha|-1)}4}\,\|(u_0,u_1)\|_{\mathcal{D}_2^{|\alpha|}}\\
\nonumber
    & \quad + C\int_0^{t/2} (1+t-s)^{-\frac{n+2(j+|\alpha|-1)}4}\,\|f(u(s,\cdot))\|_{L^1\cap L^2}ds \\
\label{eq:NudShi}
    & \quad + C\int_{t/2}^t (1+t-s)^{-\frac{j+|\alpha|-1}2}\,\|f(u(s,\cdot))\|_{L^2}ds.
\end{align}
Since~$2\leq p\leq n/(n-4)$ in Theorem~\ref{Thm:Shinl} we can apply~\eqref{eq:GN} with~$m=2$ and~$k=2$. Computing
\[ \theta_2(p)=\frac{n}2\left(\frac12-\frac1p\right), \qquad \theta_2(2p)=\frac{n}2\left(\frac12-\frac1{2p}\right)=\frac{n}4\left(1-\frac1p\right),\]
we conclude
\begin{align}
\label{eq:fusGNlowp2}
\|f(u(s,\cdot))\|_{L^1\cap L^2}
    & \lesssim \|u\|_{X_0(s)}^p (1+s)^{-p(\frac{n-2}4+\theta_2(p))}= \|u\|_{X_0(s)}^p (1+s)^{-\frac{p(n-1)-n}2}
\intertext{since~$\theta_2(p)<\theta_2(2p)$, whereas}
\label{eq:fusGNlow2p2}
\|f(u(s,\cdot))\|_{L^2}
    & \lesssim \|u\|_{X_0(s)}^p (1+s)^{-p(\frac{n-2}4+\theta_2(2p))} = \|u\|_{X_0(s)}^p (1+s)^{-\frac{p(n-1)-n/2}2}.
\end{align}
Using again~\eqref{eq:1ts} we can now write
\begin{align*}
\|Nu(t,\cdot)\|_{L^2}
    & \leq C (1+t)^{-\frac{n-2}4}\,\|(u_0,u_1)\|_{L^1\times L^2} \\
    & \quad + C\,\|u\|_{X_0(t)}^p \, (1+t)^{-\frac{n-2}4}\int_0^{t/2} (1+s)^{-\frac{p(n-1)-n}2}ds \\
    & \quad + C\,\|u\|_{X_0(t)}^p \, (1+t)^{-\frac{p(n-1)-n}2} \int_{t/2}^t (1+t-s)^{-\frac{n-2}4}ds,\\
\|\partial_t^j\partial_x^\alpha Nu(t,\cdot)\|_{L^2}
    & \leq C (1+t)^{-\frac{n-2+2j+2|\alpha|}4}\|(u_0,u_1)\|_{\mathcal{D}_2^{|\alpha|}} \\
    & \quad + C\,\|u\|_{X_0(t)}^p (1+t)^{-\frac{n+2(j+|\alpha|-1)}4} \int_0^{t/2} (1+s)^{-\frac{p(n-1)-n}2} \, ds \\
    & \quad + C\,\|u\|_{X_0(t)}^p (1+t)^{-\frac{p(n-1)-n/2}2} \int_{t/2}^t (1+t-s)^{-\frac{j+|\alpha|-1}2}\, ds
\end{align*}
for~$(j,|\alpha|)=(1,0)$ and~$j=0$, $|\alpha|=1,2$. Due to~\eqref{eq:pShi} the term~$(1+s)^{-\frac{p(n-1)-n}2}$ is integrable. Moreover, $(j+|\alpha|-1)/2<1$, that is, we can also estimate
\[ (1+t)^{-\frac{p(n-1)-n/2}2} \int_{t/2}^t (1+t-s)^{-\frac{j+|\alpha|-1}2}\, ds \approx (1+t)^{-\frac{p(n-1)-n/2}2+1-\frac{j+|\alpha|-1}2} \leq (1+t)^{-\frac{n+2(j+|\alpha|-1)}4} \]
using again~\eqref{eq:pShi}. Nevertheless, a new difficulty arises to estimate the integral term~$\int_{t/2}^t\ldots$ in~$Nu$ since we used here the~$(L^1\cap L^2)-L^2$ estimate, due to the lack of a \emph{suitable} $L^2-L^2$ estimate (see Remark~\ref{Rem:NR22}). If~$n\leq6$, then this difficulty is easily solved, since~$(1+t-s)^{-\frac{n-2}4}$ is not integrable over~$[t/2,t]$, being~$(n-2)/4\leq1$, therefore
\[ (1+t)^{-\frac{p(n-1)-n}2} \int_{t/2}^t (1+t-s)^{-\frac{n-2}4}ds \approx \begin{cases}
(1+t)^{-\frac{p(n-1)-n}2+1-\frac{n-2}4} & \text{if~$n\leq 5$,}\\
(1+t)^{-\frac{p(n-1)-n}2}\log (1+t) & \text{if~$n=6$.}
\end{cases} \]
In both cases the decay is controlled by~$(1+t)^{-\frac{n-2}4}$, due to~\eqref{eq:pShi}. Now let~$n\geq7$, that is, $(1+t-s)^{-\frac{n-2}4}$ is integrable over~$[t/2,t]$. We recall that we used~\eqref{eq:GN}, hence we already assumed~$p\geq2$. Therefore, we can estimate
\[ (1+t)^{-\frac{p(n-1)-n}2} \int_{t/2}^t (1+t-s)^{-\frac{n-2}4}ds \lesssim (1+t)^{-\frac{p(n-1)-n}2} \leq (1+t)^{-\frac{2(n-1)-n}2} \leq (1+t)^{-\frac{n-2}4}. \]
This concludes the proof.
\end{proof}


\begin{proof}[Proof of Theorem~\ref{Thm:parnl}]
We follow the proof of Theorem~\ref{Thm:halfnl}. Having in mind  Theorem \ref{Thm:parlin} we fix now the space~$A=\mathcal{D}_2^1$ for the data, and the
norm
\[ \|u\|_{X(t)} \doteq \sup_{0\leq \tau\leq t} \Bigl( (1+\tau)^{(\frac{n}4-\sigma)\frac1{1-\sigma}} \|u(\tau,\cdot)\|_{L^2}
+ (1+\tau)^{(\frac{n+2}4-\sigma)\frac1{1-\sigma}} \|\nabla u(\tau,\cdot)\|_{L^2} + (1+\tau) \|u_t(\tau,\cdot)\|_{L^2}\Bigr)\]
on the space
\[ X(t)=\left\{ u\in \mathcal{C}([0,t],H^1)\cap \mathcal{C}^1([0,t],L^2) \right\}. \]
We only prove~\eqref{eq:well}, being the proof of~\eqref{eq:contraction} analogous as in the proof of Theorem~\ref{Thm:halfnl}. Using
\eqref{eq:vpar}-\eqref{eq:vtpar} in~$[0,t]$, and~\eqref{eq:vxpar} if~$s\in [0,t/2]$ and~\eqref{eq:vxpar2} if~$s\in [t/2,t]$, we get
\begin{align}
\nonumber
\| \partial_t^j Nu(t,\cdot)\|_{L^2}
    & \leq C (1+t)^{-(\frac{n}4-\sigma)\frac1{1-\sigma}-j}\,\|(u_0,u_1)\|_{\mathcal{D}_2^j}\\
\label{eq:Nupar}
    & \quad + C\int_0^t (1+t-s)^{-(\frac{n}4-\sigma)\frac1{1-\sigma}-j}\,\|f(u(s,\cdot))\|_{L^1\cap L^2}ds \\
\nonumber
\| \nabla Nu(t,\cdot)\|_{L^2}
    & \leq C (1+t)^{-(\frac{n+2}4-\sigma)\frac1{1-\sigma}}\,\|(u_0,u_1)\|_{\mathcal{D}_2^1}\\
\nonumber
    & \quad + C\int_0^{t/2} (1+t-s)^{-(\frac{n+2}4-\sigma)\frac1{1-\sigma}}\,\|f(u(s,\cdot))\|_{L^1\cap L^2}ds \\
\label{eq:Nuxpar}
    & \quad + C\int_{t/2}^t (1+t-s)^{-(\frac12-\sigma)\frac1{1-\sigma}}\,\|f(u(s,\cdot))\|_{L^2}ds.
\end{align}
As in the proof of Theorem~\ref{Thm:Shinl} some difficulties arise to estimate the integral terms~$\int_{t/2}^t\ldots$ in~$Nu$ and~$\partial_t Nu$ since we
used here the~$(L^1\cap L^2)-L^2$ estimate, due to the lack of a \emph{suitable} $L^2-L^2$ estimate. In particular, this difficulty brings a loss in decay
in the estimates for~$u_t$ with respect to the linear estimate~\eqref{eq:vtpar}.
\\
Since~$2\leq p\leq n/(n-2)$ in Theorem~\ref{Thm:parnl} we can apply~\eqref{eq:GN} with~$m=2$ and~$k=1$. In this way we get
\begin{align*}
\|f(u(s,\cdot))\|_{L^1\cap L^2}
    & \lesssim \|u\|_{X_0(s)}^p (1+s)^{-p((\frac{n+2\theta_1(p)}4-\sigma)\frac1{1-\sigma})}= \|u\|_{X_0(s)}^p (1+s)^{-(p(\frac{n}2-\sigma)-\frac{n}2)\frac1{1-\sigma}}
\intertext{since~$\theta_1(p)<\theta_1(2p)$, whereas}
\|f(u(s,\cdot))\|_{L^2}
    & \lesssim \|u\|_{X_0(s)}^p (1+s)^{-p((\frac{n+2\theta_1(2p)}4-\sigma)\frac1{1-\sigma})} = \|u\|_{X_0(s)}^p (1+s)^{-(p(\frac{n}2-\sigma)-\frac{n}4)\frac1{1-\sigma}}.
\end{align*}
Using again~\eqref{eq:1ts} we can, finally, estimate
\begin{align*}
\|\partial_t^j Nu(t,\cdot)\|_{L^2}
    & \leq C (1+t)^{-(\frac{n}4-\sigma)\frac1{1-\sigma}-j}\|(u_0,u_1)\|_{\mathcal{D}_2^j} \\
    & \quad + C\,\|u\|_{X_0(t)}^p (1+t)^{-(\frac{n}4-\sigma)\frac1{1-\sigma}-j} \int_0^{t/2} (1+s)^{-(p(\frac{n}2-\sigma)-\frac{n}2)\frac1{1-\sigma}} \, ds \\
    & \quad + C\,\|u\|_{X_0(t)}^p (1+t)^{-(p(\frac{n}2-\sigma)-\frac{n}2)\frac1{1-\sigma}} \int_{t/2}^t (1+t-s)^{-(\frac{n}4-\sigma)\frac1{1-\sigma}-j} \, ds,\\
\|\nabla Nu(t,\cdot)\|_{L^2}
    & \leq C (1+t)^{-(\frac{n+2}4-\sigma)\frac1{1-\sigma}}\|(u_0,u_1)\|_{\mathcal{D}_2^1} \\
    & \quad + C\,\|u\|_{X_0(t)}^p (1+t)^{-(\frac{n+2}4-\sigma)\frac1{1-\sigma}} \int_0^{t/2} (1+s)^{-(p(\frac{n}2-\sigma)-\frac{n}2)\frac1{1-\sigma}} \, ds \\
    & \quad + C\,\|u\|_{X_0(t)}^p (1+t)^{-(p(\frac{n}2-\sigma)-\frac{n}4)\frac1{1-\sigma}} \int_{t/2}^t (1+t-s)^{-(\frac12-\sigma)\frac1{1-\sigma}}\, ds.
\end{align*}
Due to~\eqref{eq:ppar} the term~$(1+s)^{-(p(\frac{n}2-\sigma)-\frac{n}2)\frac1{1-\sigma}}$ is integrable. Moreover, $(\frac12-\sigma)\frac1{1-\sigma}<1$,
that is, we can also estimate
\begin{multline*}
(1+t)^{-(p(\frac{n}2-\sigma)-\frac{n}4)\frac1{1-\sigma}} \int_{t/2}^t (1+t-s)^{-(\frac12-\sigma)\frac1{1-\sigma}}\, ds \\
    \approx (1+t)^{-(p(\frac{n}2-\sigma)-\frac{n}4)\frac1{1-\sigma}+1-(\frac12-\sigma)\frac1{1-\sigma}} \leq (1+t)^{-(\frac{n+2}4-\sigma)\frac1{1-\sigma}}
\end{multline*}
using again~\eqref{eq:ppar}. On the other hand, since~$n\in[2,4]$ and~$\sigma\in(0,1/2)$ it holds
\[ 0 <\bigl(\frac{n}4-\sigma\bigr)\frac1{1-\sigma}\leq 1. \]
Therefore we may estimate
\begin{align*}
\int_{t/2}^t (1+t-s)^{-(\frac{n}4-\sigma)\frac1{1-\sigma}} & \lesssim (1+t)^{1-(\frac{n}4-\sigma)\frac1{1-\sigma}}\,\log(e+t),\\
\int_{t/2}^t (1+t-s)^{-(\frac{n}4-\sigma)\frac1{1-\sigma}-1} & \leq C.
\end{align*}
We remark that the term~$\log (e+t)$ only appears in the above estimate if~$n=4$. Using again~\eqref{eq:ppar} the proof of~\eqref{eq:upar}
and~\eqref{eq:utpar} immediately follows. This concludes the proof.
\end{proof}


\begin{proof}[Proof of Theorem~\ref{Thm:hypnl}]
We follow the proof of Theorem~\ref{Thm:parnl}. But now we propose some modifications to work with fractional derivatives of the solution of order~$2\sigma$. In particular, we use a Gagliardo-Nirenberg type inequality basing on fractional Sobolev spaces~\eqref{eq:GNfrac}. Let us consider the solution space
\[ X(t)=\mathcal{C}^1([0,t],H^{2\sigma})\cap\mathcal{C}([0,t],L^2) \]
with the norm
\[
\|u\|_{X(t)} \doteq \sup_{0\leq \tau\leq t} \Bigl( (1+\tau)^{\frac{n-2}{4\sigma}} \|u(\tau,\cdot)\|_{L^2} + (1+\tau)^{\frac{n}{4\sigma}} \|(u_t,\nabla u)(\tau,\cdot)\|_{L^2} + (1+\tau)^{\frac{n-2}{4\sigma}+1} \|u(\tau,\cdot)\|_{\dot{H}^{2\sigma}}\Bigr).
\]
More precisely, if~$n=2$ then due to Theorem \ref{Thm:hyplin} the
coefficient of~$\|u(\tau,\cdot)\|_{L^2}$ is~$(\log(e+\tau))^{-1}$.
Since this term brings no additional difficulties we will ignore it.
\\
Again we only prove~\eqref{eq:well}. As in the proof of
Theorems~\ref{Thm:Shinl} and~\ref{Thm:parnl} we only use the
$(L^1\cap L^2)-L^2$ estimate~\eqref{eq:vhyp} for~$Nu$, whereas we
use \eqref{eq:vthyp}-\eqref{eq:vxhyp}-\eqref{eq:vxxhyp} if~$s\in
[0,t/2]$ and \eqref{eq:vthyp2}-\eqref{eq:vxhyp2}-\eqref{eq:vxxhyp2}
if~$s\in [t/2,t]$ for estimating suitable derivatives of $Nu$. We
get
\begin{align*}
\| Nu(t,\cdot)\|_{L^2}
    & \leq C (1+t)^{-\frac{n-2}{4\sigma}}\,\|(u_0,u_1)\|_{L^1\cap L^2}\\
    & \quad + C\int_0^t (1+t-s)^{-\frac{n-2}{4\sigma}}\,\|f(u(s,\cdot))\|_{L^1\cap L^2}ds,\\
\| \partial_t^j Nu(t,\cdot)\|_{\dot{H}^\kappa}
    & \leq C (1+t)^{-\frac{n+2(j+\kappa-1)}{4\sigma}}\,\|(u_0,u_1)\|_{\mathcal{D}_2^{2j(1-\sigma)+\kappa}}\\
    & \quad + C\int_0^{t/2} (1+t-s)^{-\frac{n+2(j+\kappa-1)}{4\sigma}}\,\|f(u(s,\cdot))\|_{L^1\cap L^2}ds \\
    & \quad + C\int_{t/2}^t \|f(u(s,\cdot))\|_{L^2}ds,
\end{align*}
where either~$(j,\kappa)=(1,0)$, or~$j=0$ and~$\kappa=1, 2\sigma$. For~$m=1,2$ we will use the following fractional Gagliardo-Nirenberg type inequality:
\begin{equation}
\label{eq:GNfrac}
\|u(s,\cdot)\|_{L^{mp}}^p \lesssim \|u(s,\cdot)\|_{L^2}^{p(1-\theta_\kappa(mp))}\,\|u(s,\cdot)\|_{\dot{H}^\kappa}^{p\theta_\kappa(mp)},
\end{equation}
where~$\kappa\in(0,n/2)$ is a real number, and
\[ \theta_\kappa(q) \doteq \frac{n}\kappa\left(\frac12-\frac1q\right), \qquad m\leq q\leq \frac{nm}{n-\kappa\,m}.\]
In particular, these last inequalities for~$m=1,2$ give an interval for admissible $p\in [2, n/(n-2\,\kappa)]$ for the exponent~$p$. We shall distinguish three cases.
\\
Firstly, let~$n\geq4$, or~$n=3$ and~$\sigma\in(1/2,3/4)$. We set~$\kappa=2\sigma$. Indeed, in such a case, $2\sigma<n/2$, hence, the application of~\eqref{eq:GNfrac} gives
\begin{align*}
\|f(u(s,\cdot))\|_{L^1\cap L^2}
    & \lesssim \|u\|_{X_0(s)}^p (1+s)^{-p\,(\frac{n-2}{4\sigma}+\theta_{2\sigma}(p))} = \|u\|_{X_0(s)}^p (1+s)^{\frac{-p\,(n-1)+n}{2\sigma}}
\intertext{since~$\theta_{2\sigma}(p)<\theta_{2\sigma}(2p)$, whereas}
\|f(u(s,\cdot))\|_{L^2}
    & \lesssim \|u\|_{X_0(s)}^p (1+s)^{-p\,\frac{n-2+2\theta_{2\sigma}(2p)}{4\sigma}} = \|u\|_{X_0(s)}^p (1+s)^{\frac{-2p\,(n-1)+n}{4\sigma}}.
\end{align*}
Here we put~$H^0=\dot{H}^0=L^2$. Using again~\eqref{eq:1ts} we can now conclude
\begin{align*}
\| Nu(t,\cdot)\|_{L^2}
    & \leq C (1+t)^{-\frac{n-2}{4\sigma}}\,\|(u_0,u_1)\|_{L^1\cap L^2}\\
    & \quad + C\,\|u\|_{X_0(t)}^p \, (1+t)^{-\frac{n-2}{4\sigma}}\,\int_0^{t/2} (1+s)^{\frac{-p\,(n-1)+n}{2\sigma}} ds,\\
    & \quad + C\,\|u\|_{X_0(t)}^p \, (1+t)^{\frac{-p\,(n-1)+n}{2\sigma}} \,\int_{t/2}^t (1+t-s)^{-\frac{n-2}{4\sigma}}\,ds\,,\\
\| \partial_t^j Nu(t,\cdot)\|_{\dot{H}^\kappa}
    & \leq C (1+t)^{-\frac{n}{4\sigma}}\,\|(u_0,u_1)\|_{\mathcal{D}_2^{2j(1-\sigma)+\kappa}}\\
    & \quad + C\,\|u\|_{X_0(t)}^p \,(1+t)^{-\frac{n}{4\sigma}}\,\int_0^{t/2} (1+s)^{\frac{-p\,(n-1)+n}{2\sigma}} ds \\
    & \quad + C\,\|u\|_{X_0(t)}^p \,(1+t)^{\frac{-2p\,(n-1)+n}{4\sigma}}\,\int_{t/2}^t 1 ds.
\end{align*}
As usual, we use~\eqref{eq:phyp} to control all the rates of decay.
In particular, since~$p(n-1)>n+2\sigma$, it follows that
\[ (1+t)^{\frac{-2p\,(n-1)+n}{4\sigma}}\,\int_{t/2}^t 1 ds \approx (1+t)^{\frac{-2p\,(n-1)+n}{4\sigma}+1} \leq (1+t)^{-\frac{n}{4\sigma}}. \]
As in the proof of Theorem~\ref{Thm:Shinl} we have to pay attention
to the term~$\int_{t/2}^t\ldots$ in~$Nu$, since we used
the~$(L^1\cap L^2)-L^2$ estimate due to the lack of a
\emph{suitable} $L^2-L^2$ estimate (see Remark~\ref{Rem:NR22}).
If~$n\leq 2 +4\sigma$, then this difficulty is easily solved
because~$(1+t-s)^{-\frac{n-2}{4\sigma}}$ is not integrable
over~$[t/2,t]$. Therefore,
\[ (1+t)^{\frac{-p\,(n-1)+n}{2\sigma}} \,\int_{t/2}^t (1+t-s)^{-\frac{n-2}{4\sigma}}\,ds \approx \begin{cases}
(1+t)^{\frac{-p\,(n-1)+n}{2\sigma}+1-\frac{n-2}{4\sigma}} & \text{if~$n\leq 2 +4\sigma$, and~$(n,\sigma)\neq(5,3/4)$,}\\
(1+t)^{\frac{-p\,(n-1)+n}{2\sigma}}\log (1+t) & \text{if~$n=5$ and~$\sigma=3/4$.}
\end{cases} \]
In both cases the decay is controlled by~$(1+t)^{-\frac{n-2}{4\sigma}}$, due to~\eqref{eq:phyp}. Now let~$n>2+4\sigma$, this implies that $(1+t-s)^{-\frac{n-2}{4\sigma}}$ is integrable over~$[t/2,t]$. We recall that we used~\eqref{eq:GNfrac}, hence, we already assumed~$p\geq2$. For this reason we can estimate
\[ (1+t)^{\frac{-p\,(n-1)+n}{2\sigma}} \,\int_{t/2}^t (1+t-s)^{-\frac{n-2}{4\sigma}}\,ds \lesssim (1+t)^{\frac{-p\,(n-1)+n}{2\sigma}} \leq (1+t)^{-\frac{2(n-1)-n}{2\sigma}} \leq (1+t)^{-\frac{n-2}{2\sigma}}. \]
Now we come back to the other two cases. If~$n=2$, then it is sufficient to apply classical Gagliardo-Nirenberg inequality~\eqref{eq:GN} for~$k=1$. In this case we use in the right-hand side of~\eqref{eq:well} and~\eqref{eq:contraction} the space~$X_0(t)=\mathcal{C}([0,t],H^1)$ with the norm
\[
\|u\|_{X_0(t)} \doteq \sup_{0\leq \tau\leq t} \Bigl( (1+\tau)^{\frac{n-2}{4\sigma}} \|u(\tau,\cdot)\|_{L^2} + (1+\tau)^{\frac{n}{4\sigma}} \|\nabla u(\tau,\cdot)\|_{L^2} \Bigr).
\]
Under this choice we obtain the range of admissible $p\in [2,\infty)$ and the rest of the proof is analogous. It remains to consider the case~$n=3$ and~$\sigma\in[3/4,1)$. For any~$p\in[2,\infty)$ there exists~$\kappa\in(1,3/2)$ such that~$p<3/(3-2\kappa)$. This allows to apply~\eqref{eq:GNfrac} for $m=2$. In this case we use in the right-hand side of~\eqref{eq:well} and~\eqref{eq:contraction} the space~$X_0(t)=\mathcal{C}([0,t],H^\kappa)$ with the norm
\[
\|u\|_{X_0(t)} \doteq \sup_{0\leq \tau\leq t} \Bigl( (1+\tau)^{\frac{n-2}{4\sigma}} \|u(\tau,\cdot)\|_{L^2} + (1+\tau)^{\frac{n+2(\kappa-1)}{4\sigma}} \|u(\tau,\cdot)\|_{\dot{H}^{\kappa}}\Bigr).
\]
Indeed, it is clear that one can obtain an estimate for the~$\dot{H}^\kappa$ norm of the solution in Theorem~\ref{Thm:hypnl}.
\\
This concludes the proof.
\end{proof}

%

\section{Application of test function method}\label{sec:test}

One may ask if the exponents given by
\eqref{eq:phalf}-\eqref{eq:pShi}-\eqref{eq:ppar}-\eqref{eq:phyp} are
really \emph{critical}, that is, this condition is necessary for the
global existence of small data solutions. The answer is positive in
the case $\sigma\in(0,1/2]$, we interpret those models as
\emph{parabolic}-type models. The proof bases on the application of
the test function method~\cite{Z}.
\\
On the other hand, in the case $\sigma\in(1/2,1]$, we interpret those models as \emph{hyperbolic}-type models, the application of the test function method only proves that there exists in general no global in time solution if~$1<p\leq 1+2/(n-1)$, whereas the bound for the exponent given by~\eqref{eq:pShi} and~\eqref{eq:phyp} is~$p>1+(1+2\sigma)/(n-1)$.
\\
Nevertheless, we remark that this effect also appears if one
considers the heat equation~$u_t-\triangle u=|u|^p$ (which is a
\emph{parabolic} equation) and the wave equation~$u_{tt}-\triangle
u=|u|^p$ (which is a \emph{hyperbolic} equation). Indeed, for the
heat equation one has global existence of small data solutions
for~$p>1+2/n$ and there exists no global solution for a suitable
choice of the data for~$1<p\leq 1+2/n$, whereas for the wave
equation there exists a gap between the exponent obtained with the
test function method, i.e.~$1<p\leq 1+2/(n-1)$ and the exponent for
which one can prove global existence of small data solutions, that
is, $p>\gamma(n)$, where~$\gamma(n)$ is the positive root of the
equation $n(p-1)/2=(p+1)/p$.
\\
On the other hand, we recall that~$1+2/(n-1)$ is the lowest power
for the global existence of small amplitude solutions to the wave
equation with a nonlinear term $f(\partial u, \partial^2 u)$
(see~\cite{strauss}).
\begin{Thm}\label{Thm:blow}
We consider the Cauchy problem
\begin{equation}
\label{eq:blow}
\begin{cases}
u_{tt}-\Delta u+\mu (-\Delta)^{\sigma} u_t=|u|^p, & t\geq0, \ x\in\R^n,\\
u(0,x)=u_0(x), \\
u_t(0,x)=u_1(x),
\end{cases}
\end{equation}
for~$\sigma\in(0,1]$. Let us assume that the data~$u_0,u_1\in\mathcal{C}_0^\infty(\R^n)$ satisfy
\begin{equation}\label{eq:blowdata}
\int_{\Rn} \big(u_1(x) + \mu (-\triangle)^\sigma u_0(x) \big) dx >0.
\end{equation}
If
\begin{equation}\label{eq:pblow}
1<p\leq \begin{cases}
1+ \frac2{n-2\sigma} & \text{if~$\sigma\in(0,1/2]$,}\\
1+ \frac2{n-1} & \text{if~$\sigma\in[1/2,1]$,}
\end{cases}
\end{equation}
then there exists no global, sufficiently regular, non-negative solution to~\eqref{eq:blow}.
\end{Thm}
\begin{Rem} \label{Rem6}
{}From the point of view of the test function method the \emph{essential part} of~\eqref{eq:blow} is given by~$\mu (-\Delta)^{\sigma} u_t-\triangle u$ if~$\sigma\in(0,1/2]$ and by~$u_{tt}-\triangle u$ if~$\sigma\in[1/2,1]$. The other term is not relevant to determine the exponent of nonexistence given by~\eqref{eq:pblow}.
\\
If~$\sigma\in(0,1/2]$, then Theorem~\ref{Thm:blow} shows that the
exponents in~\eqref{eq:phalf} and~\eqref{eq:ppar} are
\emph{critical}. On the other hand, if~$\sigma\in(1/2,1]$, we see a
gap. If~$p>1+(1+2\sigma)/(n-1)$, then Theorems~\ref{Thm:Shinl}
and~\ref{Thm:parnl} give us existence of small data solution.
If~$1<p \leq  1+ 2/(n-1)$ we have no global existence in time. Since
the test function method does not yield an optimal result for
\emph{hyperbolic}-like models one could try to apply other methods,
for example, the functional method to close the gap.
\end{Rem}
\begin{proof}
The key tool for applying test function method with pseudo-differential operators like~$(-\triangle)^\sigma$ is the following result in Prop.2.3 of~\cite{Cordoba2} or Prop.3.3 of~\cite{Ju}. It is also used in~\cite{FK}. This result gives
\begin{equation}\label{eq:sigmaphi}
(-\triangle)^\sigma \phi^\ell \leq \ell \phi^{\ell-1} (-\triangle)^\sigma \phi
\end{equation}
for all~$\sigma\in(0,1]$ and~$\ell\geq1$, for any sufficiently regular, non-negative, decaying at infinity function~$\phi$ from the Schwartz space. 
\\
In facts, if~$\sigma=1$, then Theorem~\ref{Thm:blow} can be proved more easily. Nevertheless, for the sake of brevity, we give a proof which works for both, for the pseudo-differential case~$\sigma\in(0,1)$ and the differential case~$\sigma=1$ as well.
\\
We will use~\eqref{eq:sigmaphi} for~$\ell=p'+1$ together with
Young's inequality
\begin{equation}\label{eq:eyoung}
ab \leq \frac{a^p}p + \frac{b^{p'}}{p'}, \qquad \text{where~$p'$ is the conjugate of~$p$}.
\end{equation}
Let~$\eta\in\mathcal{C}_0^\infty([0,\infty),[0,1]),\,\phi\in\mathcal{C}_0^\infty(\Rn,[0,1])$ be such that $\eta(t) =1$ for any~$t\in [0,1/2]$ and $\eta(t)=0$
for any~$t\geq1$, $\phi(x) =1$ for any~$|x|\leq 1/2$ and $\phi(x)=0$ for any~$|x|\geq1$. We choose $\eta$ and $\phi$ such that
\[ \frac{\eta'(t)^2}{\eta(t)} + |\eta''(t)| \leq C \quad \text{for any~$t\in [1/2,1]$, \quad and} \quad \frac{|\nabla \phi(x)|^2}{\phi(x)} +
|\triangle\phi(x)| \leq C \quad \text{for any~$|x|\in [1/2,1]$,} \]
and~$\phi(x)$ is sufficiently regular so that~\eqref{eq:sigmaphi} holds. Moreover, we assume that~$\eta(t)$ is a decreasing function and that~$\phi=\phi(|x|)$
is a radial function with~$\phi(|x|)\leq\phi(|y|)$ for any~$x,y$ such that~$|x|\geq|y|$.
\\
Let $R$ be a large parameter in $[0,\infty)$. We define the test function
\[ \psi_R(t,x)\doteq \eta_R(t)\phi_R(x), \]
where
\[ \phi_R(x)\doteq \begin{cases}
\phi\left(Kx/R\right), & \text{if~$\sigma\in(0,1/2]$,} \\
\phi\left(x/R\right), & \text{if~$\sigma\in(1/2,1]$,}
\end{cases} \qquad \eta_R(t)\doteq \begin{cases}
\eta\left(t/R^{2(1-\sigma)}\right), & \text{if~$\sigma\in(0,1/2]$,} \\
\eta\left(t/R\right), & \text{if~$\sigma\in(1/2,1]$,}
\end{cases} \]
being~$K\geq1$ another large parameter that we will fix later.
\\
First, let us consider the case~$\sigma\in(0,1/2]$. Let $p'$ be the dual of $p$ and
\[ I_R\doteq \int_0^\infty\int_{\Rn} |u(t,x)|^p\psi_R(t,x)^{p'+1} dx\,dt = \int_{Q_R}|u(t,x)|^p\psi_R(t,x)^{p'+1} \,d(t,x), \]
where
\[ Q_R\doteq [0,R^{2(1-\sigma)}]\times B_{R/K}, \quad B_{R/K}\doteq\{x\in \R^n: |x|\le R/K\}.\]
Here~$u=u(t,x)$ is the global solution to~\eqref{eq:blow}. In particular, the function~$I_R$ is defined for any~$R>0$.
\\
Since~$\psi_R(t,x)\leq\psi_{S}(t,x)$ for any~$R\leq S$, we can apply Beppo-Levi convergence theorem. For this reason there exists
\[ I \doteq \lim_{R\to\infty} I_R = \int_0^\infty\int_{\Rn} |u(t,x)|^p dx\,dt \in [0,\infty].\]
We claim that~$I=0$ if~\eqref{eq:pblow} holds. Indeed, if~$I=0$, then~$u(t,x)=0$ for any~$(t,x)\in[0,\infty)\times\Rn$, therefore the proof follows from our claim.
\\
We multiply the equation in~\eqref{eq:blow} by~$\psi_R^{p'+1}$. Then, integrating by parts and by using the properties of Fourier transform, we get
\begin{align*}
0 \leq I_R
    & = -J_0 + J_1+J_2+J_3,\quad \text{where}\\
J_0
    & = \int_{B_{R/K}} \big(u_1(x) + \mu\,(-\triangle)^\sigma u_0(x) \big) \phi_R(x)^{p'+1} dx, \\
J_1
    & = \int_{Q_R} u(t,x) \, \phi_R(x)^{p'+1} \,\partial_t^2\bigl(\eta_R(t)^{p'+1}\bigr)\,d(t,x),\\
J_2
    & = -\int_0^R \partial_t\bigl(\eta_R(t)^{p'+1}\bigr) \,\int_{\Rn} u(t,x) \mu\,(-\triangle)^\sigma \bigl(\phi_R(x)^{p'+1}\bigr) dx\,dt, \\
J_3
    & = -\int_{Q_R} u(t,x) \,\eta_R(t)^{p'+1}\,\Delta\bigl(\phi_R^{p'+1}\bigr)\,d(t,x).
\end{align*}
In particular, to define~$J_2$ we used the Parseval's formula
    \[
    \int_{\R^n} v(x) (-\triangle)^\sigma u(x) dx = \int_{\R^n} \widehat{v}(\xi) |\xi|^{2\sigma} \widehat{u}(\xi) d\xi= \int_{\R^n} u(x) (-\triangle)^\sigma v(x) dx
    \]
which holds for any~$u,v \in H^{2\sigma}$. We remark that in~$J_2$
it appears the integral over~$[0,R]\times\Rn$ since
$(-\triangle)^\sigma$ is a non-local operator. To deal with it we
will use~\eqref{eq:sigmaphi}. We introduce the notations
\begin{gather*}
\hat{Q}_{R,t}\doteq [R^{2(1-\sigma)}/2, R^{2(1-\sigma)}]\times B_{R/K},\quad \hat{Q}_{R,x}\doteq[0,R^{2(1-\sigma)}]\times (B_{R/K}\setminus B_{R/(2K)}),\\
I_{R,t}\doteq\int_{\hat{Q}_{R,t}}|u(t,x)|^p\psi_R^{p'+1}(t,x)\,d(t,x), \quad I_{R,x}\doteq\int_{\hat{Q}_{R,x}}|u(t,x)|^p\psi_R^{p'+1}(t,x)\,d(t,x).
\end{gather*}
We can easily compute the measures of the sets $Q_R, \hat{Q}_{R,t}, \hat{Q}_{R,x}$ obtaining for all of them
\begin{equation}\label{eq:measQ}
|Q_R|, |\hat{Q}_{R,t}|, |\hat{Q}_{R,x}| \approx R^{2(1-\sigma)+n} K^{-n}.
\end{equation}
We shall estimate $J_1, J_2$ and $J_3$, respectively. Since
\[ \bigl|\partial_t^2 \bigl(\eta(t)^{p'+1}\bigr)\bigr| \leq \eta(t)^{p'}\,, \]
by using H\"older's inequality we can now estimate
\[ |J_1| \lesssim R^{-4(1-\sigma)} \left(\int_{\hat{Q}_{R,t}} |u(t,x)|^p \, \eta_R^{pp'}(t) \phi_{R,K}^{p\,(p'+1)} \,d(x,t) \right)^{\frac1p} \left( \int_{\hat{Q}_{R,t}} 1 d(t,x) \right)^{\frac1{p'}} \lesssim R^{-4(1-\sigma)}\, I_{R,t}^{\frac1p}\,|\hat{Q}_{R,t}|^{\frac1{p'}}.\]
Indeed, $pp' = (p')^2/(p'-1) \geq p'+1$ so that~$\eta(t)^{pp'}\leq\eta(t)^{p'+1}$, and~$p\,(p'+1)> p'+1$, hence~$\phi_{R,K}^{p\,(p'+1)}\leq \phi_{R,K}^{p'+1}$. Analogously, we can estimate
\[ |J_3|\lesssim (R/K)^{-2}\,I_{R,x}^{\frac1p}\,|\hat{Q}_{R,x}|^{\frac1{p'}}.\]
Using~\eqref{eq:measQ} we obtain
\begin{align*}
|J_1| & \lesssim I_{R,t}^{\frac1p}\,R^{-4(1-\sigma)+\frac{2(1-\sigma)+n}{p'}}\,K^{-\frac{n}{p'}}, \\
|J_3| & \lesssim I_{R,x}^{\frac1p}\,R^{-2+\frac{2(1-\sigma)+n}{p'}}\,K^{2-\frac{n}{p'}}.
\end{align*}
We remark that the estimate for~$J_1$ is better than the estimate for~$J_3$ since~$4(1-\sigma)\geq2$.
\\
We now focus our attention to~$J_2$. By using~\eqref{eq:sigmaphi}
and~\eqref{eq:eyoung} with
\[ a= |u(t,x)| \,\eta_R(t)^{p'+1/2}\phi_R^{p'-\frac1{p'}}, \quad b= C\,(p'+1)\,R^{-2(1-\sigma)} \,\phi_R^{\frac1{p'}} \, \mu\,|(-\triangle)^\sigma\phi_R(x)|, \]
and after defining
\[ C'\doteq \frac1{p'} \left( C\,(p'+1) \mu\right)^{p'}  \]
using~$u(t,x)\geq0$ together with
\[ 0\leq -\partial_t\bigl(\eta_R(t)^{p'+1}\bigr) \leq C\,\eta_R(t)^{p'+\frac12} \,, \]
we can estimate
\begin{align*}
J_2= & -\int_0^R \partial_t\bigl(\eta_R(t)^{p'+1}\bigr) \,\int_{\Rn} u(t,x) \mu\,(-\triangle)^\sigma \bigl(\phi_R(x)^{p'+1}\bigr) dx\,dt\\
& \leq -(p'+1)\int_0^R \partial_t\bigl(\eta_R(t)^{p'+1}\bigr) \,\int_{\Rn} u(t,x) \phi_R(x)^{p'}\mu\,(-\triangle)^\sigma \phi_R(x) dx\,dt\\
    & \leq C(p'+1) R^{-2(1-\sigma)} \int_{\hat{Q}_{R,t}} |u(t,x)| \, \eta_R^{p'+\frac12} (t) \phi_R(x)^{p'} \mu\,|(-\triangle)^\sigma\phi_R(x)| \,d(t,x)\\
    & \leq \frac1p \int_{\hat{Q}_{R,t}} |u(t,x)|^p \, \eta_R(t)^{(p'+1/2)\,p} \phi_R(x)^{(p'-\frac1{p'})\,p}\,d(t,x) \\ & \qquad + C'\, R^{-p'\,2(1-\sigma)} \int_{\hat{Q}_{R,t}} \phi_R(x)\,|(-\triangle)^\sigma\phi_R(x)|^{p'} \,d(t,x) \\
    & \leq \frac1p \int_{\hat{Q}_{R,t}} |u(t,x)|^p \, \eta_R(t)^{(p'+1/2)\,p}  \phi_R(x)^{(p'-\frac1{p'})\,p}\,d(t,x) \\ & \qquad + C'\, R^{-(p'-1)\,2(1-\sigma)} \int_{B_{R/K}} \phi_R(x)\,|(-\triangle)^\sigma\phi_R(x)|^{p'} \,dx \\
    & \leq \frac1p \int_{\hat{Q}_{R,t}} |u|^p \, \psi_R^{p'+1}(t,x)\,d(t,x) + C'\, R^{-(p'-1)\,2(1-\sigma)}
    (R/K)^{-2\sigma p'+n} \int_{B_1} \phi(x)\,|(-\triangle)^\sigma\phi(x)|^{p'} \,dx
\end{align*}
since $p(p'+1/2)\geq pp'\geq p'+1$ (as for~$J_1$), and
\[ \bigl(p'-\frac1{p'}\bigr) p = \bigl(p'-\frac1{p'}\bigr) \, \frac{p'}{p'-1} = \frac{(p')^2-1}{p'-1} = p'+1. \]
It is clear that
\[ \int_{\Rn} \phi(x)\,|(-\triangle)^\sigma\phi(x)|^{p'} \,dx = \int_{B_1} \phi(x)\,|(-\triangle)^\sigma\phi(x)|^{p'} \,dx  \]
is finite due to the regularity of~$\phi$. Therefore we obtain
\[ J_2 \leq \frac1p I_{R,t} + C' R^{-2p'+2(1-\sigma)+n}\,K^{2\sigma p'-n}. \]
Being~$I_{R,t}\leq I_R$ it follows that
\[ I_R (1-1/p) < -J_0 + C_K R^{-2+ \frac{2(1-\sigma)+n}{p'}} (I_{R,t}^{\frac1p}+ I_{R,x}^{\frac1p}) + C' R^{-2p'+ 2(1-\sigma)+n}\,
K^{2\sigma p'-n}. \]
We distinguish two cases, recalling that~$p\leq 1+2/(n-2\sigma)$ as
in~\eqref{eq:pblow} is equivalent to~$2p'\geq 2(1-\sigma)+n$. Let us
define~$C_{K,p}\doteq C_K/(1-1/p)$.
\\
First let~$2p'> 2(1-\sigma)+n$. In such a case we set~$K=1$. By
using the assumptions~\eqref{eq:blowdata} for the data $(u_0,u_1)$
it follows that there exists a large constant~$\overline{R}\geq1$
such that
\[ J_0 \geq C' R^{-2p'+ 2(1-\sigma)+n}  \]
for any~$R\geq\overline{R}$. Since~$I_{R,t},I_{R,x}\leq I_R$ and~$p>1$, it follows that
\[ I^{1-\frac1p} = \lim_{R\to\infty} I_R^{1-\frac1p} \leq 2C_{1,p}\, \lim_{R\to\infty} R^{-2+ \frac{2(1-\sigma)+n}{p'}} =0. \]
This concludes the proof. Now let~$2p'=2(1-\sigma)+n$, that is,
\[ I_R (1-1/p) \leq -J_0 + C_K (I_{R,t}^{\frac1p}+ I_{R,x}^{\frac1p}) + C' \,K^{2\sigma p'-n}. \]
We may now compute
\[ 2\sigma p'-n = (2(1-\sigma)+n)\,\sigma -n = (1-\sigma)(2\sigma -n) <0  \]
since we are in the case~$\sigma\in(0,1/2]$, and~$n\geq2$ if~$\sigma=1/2$ (see~\eqref{eq:pblow}). Therefore we can fix~$K\geq1$ large enough to satisfy
\[ C' K^{2\sigma p'-n}< \frac12 \int_{\Rn} \left(u_1(x) + \mu\,(-\triangle)^\sigma u_0(x) \right) dx. \]
After fixing the constant~$K$ and using the assumptions~\eqref{eq:blowdata} for the data $(u_0,u_1)$ it follows that there exists a large
constant~$\overline{R}=\overline{R}(K)\geq1$ such that
\[ J_0 \geq \frac12 \int_{\Rn} \bigl(u_1(x) + \mu\,(-\triangle)^\sigma u_0(x) \bigr) dx \]
for any~$R\geq \overline{R}$. Summarizing we obtained
\[ I_R < C_{K,p} (I_{R,t}^{\frac1p}+ I_{R,x}^{\frac1p}) \]
for any~$R\geq \overline{R}$. Since~$I_{R,t},I_{R,x}\leq I_R$ and~$p>1$ we obtain that
\[ I_R^{1-\frac1p} \leq C_{K,p}, \]
therefore~$I\leq C_{K,p}$ as well. Since we proved~$I<\infty$ by the absolute continuity of the Lebesgue integral it also follows
that~$I_{R,t},I_{R,x}\to0$ as~$R\to\infty$. Then
\[ I_R\leq C_{K,p} (I_{R,t}^{\frac1p}+ I_{R,x}^{\frac1p})\to 0 \quad \text{as~$R\to\infty$ too.}\]

Now let~$\sigma\in(1/2,1]$. In this case the proof is simpler. Indeed,
\[ Q_R\doteq [0,R]\times B_R, \quad B_R\doteq\{x\in \R^n: |x|\le R\},\]
and estimating~$J_1, J_2, J_3$ we get
\begin{align*}
|J_1| & \lesssim I_{R,t}^{\frac1p}\,R^{-2+\frac{n+1}{p'}}, \\
|J_3| & \lesssim I_{R,x}^{\frac1p}\,R^{-2+\frac{n+1}{p'}},\\
J_2 & \lesssim \epsilon(p'+1) I_{R,t} + C_\epsilon' R^{-(1+2\sigma)p'+n+1}.
\end{align*}
Now we have the same estimate for~$J_1$ and~$J_3$, whereas the estimate for~$J_2$ is better, in particular~$C_\epsilon' R^{-(1+2\sigma)p'+n+1}\to 0$
as~$R\to\infty$ since~$p\leq 1+2/(n-1)$ as in~\eqref{eq:pblow} is equivalent to~$2p'\geq n+1$. The proof easily follows.
\end{proof}
\begin{Example} \label{Exam1}
Let us consider the Cauchy problem
\begin{equation}
\label{eq:blow}
\begin{cases}
u_{tt} - \Delta u + 2\,(-\Delta)^{\frac12} u_t= |u|^p, & t\geq0, \ x\in\R^n,\\
u(0,x)= 0, \\
u_t(0,x)=u_1(x).
\end{cases}
\end{equation}
Let us assume for the non-vanishing~$u_1\in \mathcal{C}_0^\infty(\R^n)$ the condition $u_1(x)\geq0$ for any~$x\in\R^n$. If we suppose
\begin{equation}\label{eq:pblow}
1<p\leq 1+ \frac2{n-1},
\end{equation}
then there exists no global, sufficiently regular, solution to~\eqref{eq:blow}.
To apply Theorem \ref{Thm:blow} we have to show that under the above assumptions any (local or global) solution to~\eqref{eq:blow} is non-negative. Indeed,
following~\cite{NR} the solution to the linear problem
\begin{equation}
\label{eq:blowlin}
\begin{cases}
v_{tt} - \Delta v + 2\,(-\Delta)^{\frac12} v_t=0, & t\geq 0, \ x\in\R^n,\\
v(0,x)= 0, \\
v_t(0,x)=v_1(x),
\end{cases}
\end{equation}
is given by
\[ v (t,x) = t \mathcal{F}^{-1} \bigl( e^{-|\xi|t} \bigr) \ast_{(x)} v_1(x), \]
where we use
\[ \mathcal{F}^{-1} \bigl( e^{-|\xi|t} \bigr)=c_n\frac{t}{(|x|^2+t^2)^{\frac{n+1}{2}}},\,\,c_n >0. \]
By virtue of Duhamel's principle the solution to~\eqref{eq:blow} may be written as
\[ u(t,x) = t \mathcal{F}^{-1} \bigl( e^{-|\xi|t} \bigr) \ast_{(x)} u_1(x) + \int_0^t (t-s) \mathcal{F}^{-1}
\bigl( e^{-|\xi|(t-s)} \bigr) \ast_{(x)} |u(s,x)|^p\,ds. \]
Since $u_1(x)\geq0$ and~$t \mathcal{F}^{-1} \bigl( e^{-|\xi|t} \bigr)\geq0$ it follows that~$u(t,x)\geq0$. The application of Theorem \ref{Thm:blow}
completes the explanation of our example.
\end{Example}


\appendix

\section{The Gagliardo-Nirenberg inequality for fractional Sobolev spaces}

To prove Theorem~\ref{Thm:hypnl} we used the following result (see, for instance, \cite{Park}):
\begin{Lem}
Let~$n\geq1$ and~$\kappa\in(0,n/2)$ be a real number. Let $q\in[2,2n/(n-2\kappa)]$ and let~$u\in H^\kappa$. Then~$u\in L^q$ and
\[ \|u\|_{L^q} \leq C(n,\kappa,q) \, \|u\|_{L^2}^{1-\theta}\,\|u\|_{\dot{H}^\kappa}^\theta,\,  \]
where
\begin{equation}\label{eq:thetakq}
\theta= \frac{n}\kappa \left(\frac12-\frac1q\right).
\end{equation}
\end{Lem}
We remark that~$\theta\in[0,1]$ and that~\eqref{eq:thetakq} is equivalent to
\begin{equation}\label{eq:thetakq2}
\theta \left(\frac12-\frac{\kappa}n\right) + \frac{1-\theta}2 = \frac1q.
\end{equation}
\begin{proof}
If~$q=2$ the statement is trivial. Let~$q\in(2,2n/(n-2\kappa)]$. We
put~$p\doteq 2/(q\,(1-\theta))$. Then $p\in(1,\infty]$ since
\[ (1-\theta)\,q = q-\frac{nq-2n}{2\kappa} = \frac{-q(n-2\kappa)+2n}{2\kappa} \in [0,2). \]
Let $p'=2/(2-q(1-\theta))\in[1,\infty)$ be its Sobolev conjugate.
Let us define~$p_0\doteq q\,\theta\,p'$. Using~\eqref{eq:thetakq2}
it follows $p_0=2n/(n-2\kappa)\in(2,\infty)$.
\\
By virtue of H\"older's inequality we are now in a position to
estimate
\[ \|u\|_{L^q}^q = \| |u|^q \|_{L^1} \leq \||u|^{q(1-\theta)}\|_{L^p} \, \||u|^{q\theta}\|_{L^{p'}} = \|u\|_{L^2}^{q(1-\theta)} \, \|u\|_{L^{p_0}}^{q\theta}. \]
Let~$f=(-\triangle)^{\frac{\kappa}2}u$, that is,
$\widehat{u}=|\xi|^\kappa\,\widehat{f}$. By using Riesz potential it
follows that
\[ u = I_\kappa f = C_\kappa\,\int_{\R^n} \frac{f(y)}{|x-y|^{n-\kappa}}\,dy, \]
therefore~$\|u\|_{L^{p_0}}\lesssim \|f\|_{L^2}=\|u\|_{\dot{H}^\kappa}$, by virtue of Hardy-Littlewood-Sobolev inequality. This concludes the proof.
\end{proof}


%
\end{document}